  \RequirePackage{fix-cm}

\documentclass[smallcondensed]{svjour3}     
\smartqed 
\usepackage[utf8]{inputenc}
\usepackage{amsmath}
\usepackage{amssymb}
\usepackage{siunitx}

\usepackage{float}
\usepackage{pgf,tikz,pgfplots}
\pgfplotsset{compat=1.14}
\usepackage{mathrsfs}
\usetikzlibrary{patterns,arrows,decorations}
\usepackage{graphicx}
\usepackage{multicol}
\usepackage{marvosym }
\usepackage{wasysym}
\usepackage{tikz}
\usepackage{arydshln}
\usepackage{multirow}

\newif\iflongversion

\longversiontrue

\def\R{\mathbb{R}}

\def\A{\mathcal{A}}

\begin{document}

\title{Stability Analysis of Interface Conditions for Ocean-Atmosphere Coupling
\thanks{This material is based upon work supported by the U.S. Department of
Energy, Office of Science, Office of Advanced Scientific Computing
Research and Office of Biological and Environmental Research, Scientific
Discovery through Advanced Computing (SciDAC) program under contract
DE-AC02-06CH11357 through the Coupling Approaches for Next-Generation
Architectures
(CANGA) project.}}
\titlerunning{Stability of Ocean-Atmosphere Coupling}

\author{
  Hong Zhang
  \and
  Zhengyu Liu
  \and
  Emil Constantinescu
  \and
  Robert Jacob
}
\authorrunning{H. Zhang, Z. Liu, E. Constantinescu, R. Jacob}

\institute{\Letter H. Zhang \at
              Mathematics and Computer Science Division\\
              Argonne National Laboratory, Lemont, IL\\
              \email{hongzhang@anl.gov}
           \and
           Z. Liu \at
              Department of Geography\\
              Ohio State University, Columbus, OH\\
              \email{liu.7022@osu.edu}
           \and
           E. Constantinescu \at
              Mathematics and Computer Science Division\\
              Argonne National Laboratory, Lemont, IL\\
              \email{emconsta@anl.gov}
           \and
           R. Jacob \at
           Environmental Science Division\\
           Argonne National Laboratory, Lemont, IL\\
           \email{jacob@anl.gov}
}
\date{Received: date / Accepted: date}

\maketitle
\vspace{-1in}
{\bf Preprint ANL/MCS-P9224-0819}
\vspace{1in}

\begin{abstract}
In this paper we analyze the stability of different coupling strategies for
multidomain PDEs that arise in general circulation models used in climate
simulations. We focus on fully coupled ocean-atmosphere models that are needed
to represent and understand the complicated interactions of these two systems,
becoming increasingly important in climate change assessment in recent years.
Numerical stability issues typically arise because of different time-stepping
strategies applied to the coupled PDE system. In particular, the contributing
factors include using large time steps, lack of accurate interface flux, and
singe-iteration coupling. We investigate the stability of the coupled
ocean-atmosphere models for various interface conditions such as the
Dirichlet-Neumann condition and the bulk interface condition, which is unique to
climate modeling. By analyzing a simplified model, we demonstrate how the
parameterization of the bulk condition and other numerical and physical
parameters affect the coupling stability.
\keywords{stability analysis, coupled system, partitioned algorithm, ocean-atmosphere}
\subclass{65M12, 34D20, 76R50}
\end{abstract}

%
\section{Introduction}
%
%
We analyze the stability of different coupling strategies for multidomain
partial differential equations (PDEs) motivated by general circulation models
used in climate simulations. Solving these problems with large time steps on
each of the domains is known to cause numerical stability issues of the coupled
PDE system. Without loss of generality we consider two coupled PDEs that
correspond to coupled ocean-atmosphere:
\begin{subequations}
\label{eq:abstract:problem}
\begin{align}
  \label{eq:abstract:problem:1}
  \frac{\partial u_1(t,x)}{\partial t} &= F_1(t,u_1(t,x),u_2(t,x)) &&\textnormal{[domain 1]}\\
  \label{eq:abstract:problem:2}
  \frac{\partial u_2(t,x)}{\partial t} &= F_2(t,u_1(t,x),u_2(t,x)) &&\textnormal{[domain 2]}\\
  \label{eq:abstract:problem:12}
  0&=G_{12}(u_1(t,x),u_2(t,x))\,, &&\textnormal{[interface 1-2]}\\
  \nonumber
  t \ge t_0\,,~ u_k \in \mathbb{R} \times \mathcal{D}_k\,,&~ F_k :
  \mathbb{R} \times \mathcal{D} \rightarrow  \mathcal{D}_k\,,~k=1,2\,\\
  \nonumber
   G_{12}: \mathcal{D} \rightarrow \mathcal{D}\,,&~
   \mathcal{D}=\bigcup_{k=1,2} \mathcal{D}_k
\end{align}
\end{subequations}
where, for instance, \eqref{eq:abstract:problem:1} represents the atmosphere,
\eqref{eq:abstract:problem:2} the ocean, and \eqref{eq:abstract:problem:12}
 the interaction between the two problems. Stability is defined as $||u^n(x)||
\le C ||u(t_0,x)||$, $\forall n>0$, where $C$ is a finite constant independent
of $t$ and $u=[u_1,u_2]^\top$.

This problem has been analyzed by normal mode analysis and matrix stability. The
normal mode method was originally developed by Godunov and Ryabenkij
\cite{Godunov1963}, Kriess \cite{Kreiss1968} and Osher \cite{Osher1969}, and
later led to the theory of Gustafsson, Kriess, and Sundstrom (GKS)
\cite{Gustafsson1972}, which establishes necessary and sufficient conditions
that the discretization schemes must satisfy in order to ensure stability.
We are interested in utilizing this tool to analyze stability of various
interface conditions for coupled climate models.

\subsection{Scientific application}
Coupling methods have been a limiting factor in researchers' ability to address
science questions where the relevant processes are strongly coupled. Each model
inside the coupled system is often associated with different time scales, posing
a great challenge on time integration. The integration is becoming more
challenging as the coupled simulation codes are evolving to support high
resolution in time and space and concurrent execution of components. While the
time integration for each model is often well founded in theory, little work has
been done on characterizing the influence of different coupling strategies on
the stability of the fully coupled system.

\iflongversion
In the U.S. climate community models, including CESM \cite{Hurrell2013} and DOE
E3SM \cite{Golaz2019} (previously known as ACME), the Earth system components
(global atmosphere, global ocean, sea ice and land surface) are numerically
treated independently. The interaction among these components is done through a
component called a ``coupler'' which accommodates the information exchange. In
the E3SM coupler, the components of the Earth system model are marched forward
in time nearly independently from each other, with appropriate field information
exchanged using the lagged states at the previous time interval, also known as
``explicit flux coupling.'' While this popular coupling method, really more of a
decoupled method, enables concurrent execution of multiple components for
computational efficiency, in certain parameter regimes it has been shown to
induce numerical instability at the air-sea interface \cite{Lemarie2015},
air-land interface \cite{Beljaars2017}, and sea-ice interface
\cite{Hallberg2014,Roberts2015}. In the tropics where the atmosphere is very
sensitive to the sea surface forcing, the standard coupling strategy that
propagates the atmospheric model with multiple small explicit time steps
followed by a large ocean time step is physically reasonable
\cite{Bryan1996,Perlin2007,Bao2000}. In other regions such as the extratropics
\cite{Kushnir2002}, however, the atmospheric forcing is dominated by its own
internal variability, providing rapid feedback to the ocean via the turbulent
heat flux. Thus, using large time steps for the ocean can be problematic. On
long timescales (e.g., more than one year), the standard coupling method fails
to produce the correct ocean-atmosphere heat flux both in magnitude and in sign,
contributing to poor understanding of climate variability in these regions
\cite{Kushnir2002,BrethBattisti2000,Liu2007,Liu2004}. In \cite{Kushnir2002},
Kushnir et al. concluded that the thermal coupling coefficient should be a
function of latitude and season in order to enhance the predictability of
extratropical systems.

The occasional failure of decoupled methods can be attributed to the fact that
these methods perform a single step of an iterative process \cite{Lemarie2015},
which is insufficient to secure stable and accurate solutions. A theoretically
ideal solution for stability is to build a monolithic fully coupled system and
solve it implicitly. Doing so is difficult, however,  because it is intrusive in
nature, requiring significant  development effort to refactor existing codes;
and solving the full implicit system efficiently is also a computational
challenge. A more practical approach is for each component to compute interface
fluxes using lagged information from the other components and implicitly the
states owned by itself. In this way, partial implicitness has been added, thus
improving stability without the need for many iterations when compared with
explicit flux coupling; but the resulting system is not as stable as the one
using full implicitness \cite{Lemarie2015}. Rigorous analysis of decoupled
methods remains limited and does not provide sufficient understanding of optimal
choices for stability.

\fi

To improve the stability of surface models (e.g., snow, ice, or soil) that are
coupled to atmospheric models, Beljaars et al. \cite{Beljaars2017} proposed a
fully implicit formulation that uses an estimate of the surface temperature at
the new time level instead of solving for it. The estimate is obtained with an
empirical relation between surface heat flux and surface temperature that is
derived from idealized simulations of a fully coupled implicit system.  A matrix
stability analysis was preformed on the forced surface models; however, only an
empirical stability condition was given.

The stability limits in fluid-solid interaction (FSI) coupled models have been
widely studied, but little is known for climate models that are intrinsically
large scale and involve complex dynamics. Recently Connors and his colleagues
conducted a series of studies
\cite{Connors2009,Connors2012,Connors2011,Aggul2018} that provide rigorous
analysis of various partitioned algorithms for fluid-fluid interaction with a
focus on convergence and accuracy aspects. Peterson \cite{Peterson2019} et. al.
proposed a new synchronous partitioned method that eliminates the need to solve
a coupled implicit system, while not subject to additional stability constraints
as in traditional partitioned schemes. Although the approach is developed for
general transmission problem, it can be potentially extended to coupled climate
models.

This work studies the stability of different partitioned coupling methods for
ocean-atmosphere coupling, which involves the most computationally expensive
components in climate models. We examine the stability behaviors of common
interface conditions with a focus on the bulk interface condition that is
pervasive in coupled climate models. We show by analysis that partially implicit
and fully implicit flux coupling can lead to an unconditionally stable coupling
algorithm, whereas explicit flux coupling requires certain
Courant-Friedrichs-Lewy (CFL)-like conditions to be satisfied for stability. We
also derive the closed-form necessary and sufficient stability condition for a
one-way coupled system, and connect one-way and two-way coupled models in terms
of stability regions. The theoretical analyses are accompanied by numerical
experiments. The new results lead to better understanding of the stability
properties of the existing algorithms as well as provide guidance for developing
new stable coupling algorithms.



\subsection{Model problem}
Since the vertical diffusion provides the strongest coupling between the
atmospheric boundary layer and the ocean, we consider a 1D diffusion equation
for temperature defined on two neighboring domains
\begin{subequations}
\begin{align}
\frac{\partial}{\partial t} T_+ = \frac{\partial}{\partial z} (K_+ \frac{\partial}{\partial z} T_+
), \quad z > 0 \label{eq:1dadvdiff1} \\
\frac{\partial}{\partial t} T_- = \frac{\partial}{\partial z} (K_- \frac{\partial}{\partial z} T_-
 ), \quad z < 0
\label{eq:1dadvdiff2}
\end{align}
\end{subequations}
with the eddy diffusivity coefficient determined by
\begin{equation}
    K_\pm = \frac{\nu_\pm}{\rho_\pm c_\pm} ,
\end{equation}
where $\nu$ and $\rho$ and $c$ are the heat
diffusion coefficient, the density, and the heat capacity, respectively.
The subscripts $+$ and $-$ correspond to atmosphere and ocean. respectively. In
this case $u_1=T_+$, $u_2=T_-$, $\mathcal{D}_1$ is the positive domain, and
$\mathcal{D}_2$ is the negative domain in the \eqref{eq:abstract:problem}.

In addition to the governing equations for each subsystem, the continuity of
fluxes at the interface is typically used as the coupling condition
\begin{equation}
    \rho_+ c_+ K_+ \frac{\partial}{\partial z} T_+ = \rho_- c_- K_- \frac{\partial}{\partial z} T_-\,,
\end{equation}
which corresponds to \eqref{eq:abstract:problem:12}.

\iflongversion
Note that the diffusion problem can also be represented in an equivalent form
\begin{subequations}
\begin{align*}
\rho_+ c_+ \frac{\partial}{\partial t} T_+ = \frac{\partial}{\partial z} (\nu_+ \frac{\partial}{\partial z} T_+), \quad z > 0  \\
\rho_- c_- \frac{\partial}{\partial t} T_- = \frac{\partial}{\partial z} (\nu_- \frac{\partial}{\partial z} T_-), \quad z < 0
\end{align*}
\end{subequations}
and the flux continuity becomes
\begin{equation*}
    \nu_+ \frac{\partial}{\partial z} T_+ = \nu_- \frac{\partial}{\partial z} T_- .
\end{equation*}
This alternative form makes the coupling condition simpler, but involves more
parameters in the implementation of the subsystems.
\fi

\subsection{Interface conditions}

The physical processes in the interior of each domain are independent of the
others; each domain-specific process interacts with processes in other domains
only at their common interface. Therefore, conditions need to be imposed at the
interface in order to guarantee that the numerical solution exists for the
coupled system.

In classical domain decomposition methods,  both the state and the
normal-direction flux often are required to be continuous across the interface.
When a partitioned approach is used, each domain is solved independently by
using boundary information coming from the other domains. The information is
often lagged in time if the domains are solved concurrently. A natural choice
for implementing the interface conditions is to treat them as boundary
conditions and apply Dirichlet or Neumann conditions to each subproblem.




In a climate model, the surface fluxes of momentum, sensible heat, and latent
heat are calculated by using bulk formulas, following the Monin-Obukhov
similarity theory \cite{Yaglom_1994}. The physical processes near the surface
are parameterized as bulk transfer coefficients as studied in classical works
\cite{Liu1979,Smith1988,Fairall1996}. Through bulk transfer coefficients, a bulk
formula relates the flux to easily measured surface quantities such as averaged
wind speed, temperature, and humidity. The most popular bulk flux algorithm, the
Coupled Ocean–Atmosphere Response Experiment (COARE) algorithm
\cite{Fairall1996}, has been used in coupled air-sea \cite{Bao2000} and sea-ice
\cite{Andreas2010} simulations and has been continuously improved over wider
wind-speed range and more complicated physics such as weave feedback
\cite{Edson2013,Kudryavtsev2014}.

From a mathematical point of view, the bulk flux algorithm can be abstracted as
a special type of interface condition: the flux is continuous in the normal
direction at the interface, but the state can jump across the interface.
%
%
%
%
Without loss of generality, the bulk interface condition can be described as
\begin{equation}
    \nu_+ \frac{\partial}{\partial z} T_+ = \nu_- \frac{\partial}{\partial z} T_- = b (T_+ - T_-)
\end{equation}
with the bulk coefficient defined as
\begin{equation}
  b = \rho_+ c_+ C_H \|U\| ,
\end{equation}
where $C_H$ is the exchange coefficient at reference height $r$ (e.g., a typical
value is 10 meters) that depends on surface roughness and local stability and
$\|U\|$ is the absolute (wind) velocity. This parameter is defined in a bulk way
in the region between the lowest vertical level in the atmospheric model and the
shallowest vertical level in the oceanic model.
\iflongversion
Some authors in literature have described
the same interface condition based on a variable $\alpha$, which can be related
to $b$ by
\begin{equation}
   \alpha =  b/(\rho_+ c_+).
\end{equation}
For the convenience of denoting flux, we prefer the notation $b$. But using the
other notation should not affect the conclusions reached in this paper.   
\fi

\iflongversion
Consequently, the heat flow at the the interface can be computed as
\begin{equation}
     K_+ \frac{\partial}{\partial z} T_+ =  \rho_+ c_+ \alpha (T_+ - T_-),\quad  K_-  \frac{\partial}{\partial z} T_-  = \frac{\rho_- c_-}{\rho_+ c_+} \alpha (T_+ - T_-).
\end{equation}
\fi

\subsection{Partitioned coupling algorithms}

The classical partitioned approach is the convectional serial staggered (CSS)
algorithm that solves the two subsystems alternatively. We assume at time $t_n$
that the temperature fields $T_j^n$ are given and are consistent at the
interface. With a grid setup shown in Figure \ref{fig:dn_grid}, the algorithm
can be described compactly as
\begin{subequations}
  \begin{align}
  & T_{0_+}^n = T_{0_-}^n \label{eq:css1}\\
  & \rho_+ c_+ \frac{\partial}{\partial t} T_j = \frac{\nu_+}{(\Delta z_+)^2}  \left( T_{j+1} -2 T_j + T_{j-1}\right), \quad j>0 \label{eq:css2} \\
  & q_{\frac12} = \nu_+  \frac{T_1 - T_{0_+}}{\Delta z_+}, \quad q_{-\frac12} = \nu_-  \frac{T_{0_-}- T_{-1}}{\Delta z_-} \label{eq:css3} \\
  & \frac12 \left( \rho_- c_- \Delta z_- + \rho_+ c_+ \Delta z_+ \right) \frac{\partial}{\partial t} T_{0_-} = q_{\frac12} - q_{-\frac12},  \quad j=0 \label{eq:css4} \\
  & \rho_- c_- \frac{\partial}{\partial t} T_j = \frac{\nu_-}{(\Delta z_-)^2}  \left( T_{j+1} -2 T_j + T_{j-1}\right) \quad j<0 . \label{eq:css5}
  \end{align}
  \label{eq:css}
\end{subequations}
\begin{figure}[ht]
  \centering
  \resizebox{0.95\textwidth}{!}{
    \definecolor{sexdts}{rgb}{0.1803921568627451,0.49019607843137253,0.19607843137254902}
    \definecolor{rvwvcq}{rgb}{0.08235294117647059,0.396078431372549,0.7529411764705882}
  \definecolor{cqcqcq}{rgb}{0.7529411764705882,0.7529411764705882,0.7529411764705882}
  \tikzstyle{every node}=[font=\Large]
  \begin{tikzpicture}[line cap=round,line join=round,>=triangle 45,x=1cm,y=1cm]
  \draw [line width=1pt] (-7,0)-- (10,0);
  \draw [line width=1pt,dash pattern=on 8pt off 8pt] (10,0)-- (12,0);
  \draw [->,line width=1pt] (12,0) -- node [above=4pt]{z} ++ (2,0);

  \draw [fill=orange] (-7,0) circle (4pt); \node [color=orange,above=4pt] at (-7,0) {-4};
  \draw [fill=orange] (-5,0) circle (4pt); \node [color=orange,above=4pt] at (-5,0) {-3};
  \draw [fill=orange] (-3,0) circle (4pt); \node [color=orange,above=4pt] at (-3,0) {-2};
  \draw [fill=orange] (-1,0) circle (4pt); \node [color=orange,above=4pt] at (-1,0) {-1};
  \draw [fill=orange] (1,0) circle (4pt); \node [color=orange,above right=4pt] at (1,0) {0};

  \draw [line width=1pt] (1,1)-- (1,-1);

  \draw [fill=rvwvcq] (4,0)  ++(-4pt,0 pt) -- ++(4pt,4pt)--++(4pt,-4pt)--++(-4pt,-4pt)--++(-4pt,4pt); \node [color=rvwvcq,above=4pt] at (4,0) {1};
  \draw [fill=rvwvcq] (7,0)  ++(-4pt,0 pt) -- ++(4pt,4pt)--++(4pt,-4pt)--++(-4pt,-4pt)--++(-4pt,4pt); \node [color=rvwvcq,above=4pt] at (7,0) {2};
  \draw [fill=rvwvcq] (10,0)  ++(-4pt,0 pt) -- ++(4pt,4pt)--++(4pt,-4pt)--++(-4pt,-4pt)--++(-4pt,4pt); \node [color=rvwvcq,above=4pt] at (10,0) {3};

  \draw [decorate,decoration={brace,amplitude=10pt,mirror,raise=4pt},yshift=-2pt](-7,0) -- (-5,0) node [black,midway,yshift=-0.8cm] {$\Delta z_-$};
  \draw [decorate,decoration={brace,amplitude=10pt,mirror,raise=4pt},yshift=-2pt](7,0) -- (10,0) node [black,midway,yshift=-0.8cm] {$\Delta z_+$};
  \end{tikzpicture}
  }
  \caption{Grid setting for Dirichlet-Neumann condition. Space is discretized on
  uniform grids with each node denoted by $z_j = j \Delta z_+ \quad j=0,1,2,\dots
  $ and $z_j = j \Delta z_- \quad j=0,-1,-2,\dots$ where $\Delta z_+$ and $\Delta
  z_-$ are the grid spacing on the left and right sides of the interface
  respectively. They may differ significantly.}
  \label{fig:dn_grid}
\end{figure}
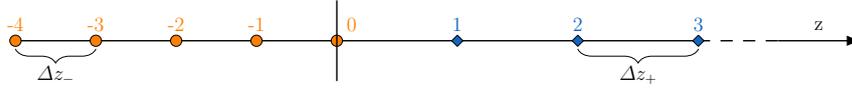

The procedure to advance the solution from $t_n$ to $t_{n+1}$ is as follows:
\begin{enumerate}
    \item At time $t_n$, transfer the temperature $T_{0_-}^n$ to the positive
    domain, and update its interface boundary with a Dirichlet condition
    \eqref{eq:css1}.
    \item Advance the solution on the right domain from $t_n$ to $t_{n+1}$
    \eqref{eq:css2},  and compute the flux at the interface \eqref{eq:css3}.
    \item Transfer the flux quantifies to the left domain, and update its
    interface boundary with a Neumann condition \eqref{eq:css4}.
    \item Advance the solution on the left domain \eqref{eq:css5}
\end{enumerate}


This algorithm can also be applied to Neumann-Neumann or bulk interface
conditions with a slightly different grid setting, as shown in Figure
\ref{fig:bulk_grid}. For the convenience of flux calculation, the interface is
positioned at the cell edge for both domains; in this way, each subdomain can
have uniform cells, but ghost cells need to be used to apply the boundary
conditions at the interface. We denote the ghost cell values for the left and
the right domains by $T_{0_+}$ and $T_{0_-}$, respectively.
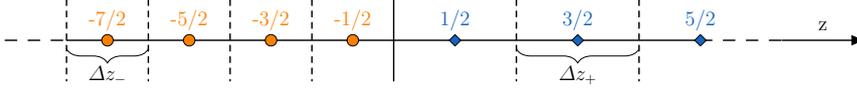
\begin{figure}[ht]
\centering
\resizebox{0.95\textwidth}{!}{
\definecolor{rvwvcq}{rgb}{0.08235294117647059,0.396078431372549,0.7529411764705882}
\definecolor{cqcqcq}{rgb}{0.7529411764705882,0.7529411764705882,0.7529411764705882}
\tikzstyle{every node}=[font=\Large]
\begin{tikzpicture}[line cap=round,line join=round,>=triangle 45,x=1cm,y=1cm]
\draw [line width=1pt,dash pattern=on 8pt off 8pt] (-7.5,0)-- (-6,0);
\draw [line width=1pt] (-6,0)-- (9.5,0);
\draw [line width=1pt,dash pattern=on 8pt off 8pt] (9.5,0)-- (11.5,0);
\draw [->,line width=1pt] (11.5,0) -- node [above=4pt]{z} ++ (2,0);

\draw [line width=1pt,dash pattern=on 4pt off 4pt] (-6,1)-- (-6,-1);
\draw [fill=orange] (-5,0) circle (4pt); \node [color=orange,above=4pt] at (-5,0) {-7/2};
\draw [line width=1pt,dash pattern=on 4pt off 4pt] (-4,1)-- (-4,-1);
\draw [fill=orange] (-3,0) circle (4pt); \node [color=orange,above=4pt] at (-3,0) {-5/2};
\draw [line width=1pt,dash pattern=on 4pt off 4pt] (-2,1)-- (-2,-1);
\draw [fill=orange] (-1,0) circle (4pt); \node [color=orange,above=4pt] at (-1,0) {-3/2};
\draw [line width=1pt,dash pattern=on 4pt off 4pt] (0,1)-- (0,-1);
\draw [fill=orange] (1,0) circle (4pt); \node [color=orange,above=4pt] at (1,0) {-1/2};

\draw [line width=1pt] (2,1)-- (2,-1);

\draw [fill=rvwvcq] (3.5,0)  ++(-4pt,0 pt) -- ++(4pt,4pt)--++(4pt,-4pt)--++(-4pt,-4pt)--++(-4pt,4pt); \node [color=rvwvcq,above=4pt] at (3.5,0) {1/2};
\draw [line width=1pt,dash pattern=on 4pt off 4pt] (5,1)-- (5,-1);
\draw [fill=rvwvcq] (6.5,0)  ++(-4pt,0 pt) -- ++(4pt,4pt)--++(4pt,-4pt)--++(-4pt,-4pt)--++(-4pt,4pt); \node [color=rvwvcq,above=4pt] at (6.5,0) {3/2};
\draw [line width=1pt,dash pattern=on 4pt off 4pt] (8,1)-- (8,-1);
\draw [fill=rvwvcq] (9.5,0)  ++(-4pt,0 pt) -- ++(4pt,4pt)--++(4pt,-4pt)--++(-4pt,-4pt)--++(-4pt,4pt); \node [color=rvwvcq,above=4pt] at (9.5,0) {5/2};

\draw [decorate,decoration={brace,amplitude=10pt,mirror,raise=4pt},yshift=-2pt](-6,0) -- (-4,0) node [black,midway,yshift=-0.8cm] {$\Delta z_-$};
\draw [decorate,decoration={brace,amplitude=10pt,mirror,raise=4pt},yshift=-2pt](5,0) -- (8,0) node [black,midway,yshift=-0.8cm] {$\Delta z_+$};
\end{tikzpicture}
} \caption{Grid setting for bulk and Neumann-Neumann interface conditions. Space
is discretized on uniform grids with each node denoted by $z_{j+1/2} = (j+1/2)
\Delta z_+ \quad j=0,1,2,\dots $ and $z_{j+1/2} = (j+1/2) \Delta z_- \quad
j=-1,-2,\dots$ where $\Delta z_+$ and $\Delta z_-$ are the grid spacing on the
left and right sides of the interface respectively. They may differ
significantly.}
\label{fig:bulk_grid}
\end{figure}

After spatial discretization, the following system of ordinary differential
equations is obtained:
\begin{subequations}
  \begin{align}
  & q_{0_+} = b (T_{\frac12} - T_{-\frac12} ) \label{eq:fv_mon1} \\
  & \frac{\partial}{\partial t} T_{j+\frac12} = \frac{1}{\rho_+ c_+ \Delta z_+}  \left( q_{j+1} - q_{j} \right), \quad j \geq 0 \label{eq:fv_mon2} \\
  & q_{0_-} = q_{0_+} \label{eq:fv_mon3} \\
  & \frac{\partial}{\partial t} T_{j+\frac12} = \frac{1}{\rho_- c_- \Delta z_-}  \left(  q_{j+1} - q_{j}  \right), \quad j<0 . \label{eq:fv_mon4}
  \end{align}
  \label{eq:css_bulk}
\end{subequations}

For simplicity, we will focus on backward Euler when referring to the implicit
time stepping method, unless otherwise noted. Extending to high-order time
integration schemes is possible but complicates the algebraic analysis
considerably. 

\subsection{Manuscript organization}

The rest of this paper is organized as follows. Normal mode stability analysis
for two-way coupled models using Dirichlet-Neumann and bulk interface conditions
is given in Sections \ref{sec:dn} and \ref{sec:bulk}. Section
\ref{sec:forced_model} applies the analysis to a one-way coupled model that can
be considered as a limiting case for two-way coupled models. Section
\ref{sec:validation} describes how to validate the analysis using matrix
stability and provides the numerical results. In Section 6, we summarize the
results and their implications to practical applications.


\section{Normal mode stability analysis for Dirichlet-Neumann condition}\label{sec:dn}
In the normal mode method, the PDE solution is represented as a certain mode,
which is typically exponential in time and space. For example, the numerical
solution of one-dimensional PDE can be assumed in the form
\begin{equation}
T_j^b = \A^n \phi_j,
\end{equation}
where $\A$ is an eigenvalue with eigenfunction $\phi_j$ (bounded as $j$ goes to
infinity). Then one needs to solve a system of algebraic equations, resulting
directly from numerical discretization and boundary (or interface) conditions,
for the amplitude of the mode.

Based on this technique, the GKS theory \cite{Gustafsson1972} establishes a
necessary and sufficient condition for stability, as summarized below.
\begin{proposition}
  The approximation scheme is stable in the GKS sense if and only if no
  nontrival eigensolution exists for $|\A|\geq 1$.
  \label{prop:GKS}
\end{proposition}

The stability definition is given in Definition 3.3 in \cite{Gustafsson1972},
and  Proposition \ref{prop:GKS} refers to Theorem 5.1 in Section 5 in
\cite{Gustafsson1972}. Although the original theory was described for
quarter-plane problems, it can be naturally applied to problems with multiple
boundaries and interface problems such as FSI.

Giles has shown in \cite{Giles1997} that for FSI problems that the key point to
achieving numerical stability is to transfer the interface temperature from the
solid to the fluid (Dirichlet condition) and to pass the heat flux (Neumann
condition) from the fluid to the solid. Roux and Garaud confirmed in
\cite{Roux2009} that Dirichlet conditions must be imposed in the domain with
lower conductivity.

Lemarie et al. also followed Giles' partitioned approach in \cite{Lemarie2015}.
The main difference between Giles' approach and the one considered in this paper
is the former is derived based on an engineering consideration---a half-sized
cell is used at the interface so that each domain can be treated independently.
Consequently, the partitioned approach solves a slightly different model from
the one targeted by a monolithic approach. In this section, we present a normal
mode analysis for a partitioned approach without using half-sized cell. Both
explicit and implicit time-stepping schemes will be considered in order to
provide a direct comparison with  results in \cite{Giles1997}.

\subsection{Explicit flux coupling with explict time stepping}

Following the explicit methods (as opposed to implicit) of Giles
\cite{Giles1997}, we calculate the interfacial flux \eqref{eq:css3} for the time
step from $t_n$ with
\begin{equation}
  q_{\frac12}^{n} = \nu_+  \frac{T_1^n - T_{0_+}^n}{\Delta z_+},  \quad q_{-\frac12}^n = \nu_-  \frac{T_{0_-}^n - T_{-1}^n}{\Delta z_-}
  \label{eq:ex_interfacial_flux}
\end{equation}
and assume a solution of the form
\begin{align}
  \label{eq:sol_form_Giles}
T^n_j = \begin{cases}
\A^n \kappa_-^j, & j = 0, -1, -2, \ldots \\
\A^n \kappa_+^j, & j = 1,2, \ldots,
\end{cases}
\end{align}
where $j$ denotes exponent, $\kappa^j$ are bounded functions of space, and $\A$ is
a complex scalar. We define the following constants:
\begin{equation}
r = \frac{\rho_+ c_+ \Delta z_+}{\rho_- c_- \Delta z_-}, \quad \quad d_\pm = \frac{\nu_\pm \Delta t}{\rho_\pm c_\pm (\Delta z_\pm)^2}.
\label{eq:def}
\end{equation}

Discretizing \eqref{eq:css} explicitly with the coupling flux being computed
according to the formula \eqref{eq:ex_interfacial_flux}, substituting the
solution \eqref{eq:sol_form_Giles} into the discretized equations, and using
\eqref{eq:def} to simplify expression, we have
\begin{subequations}
  \begin{align}
  \A &= 1+d_+(\kappa_+-2+\kappa_+^{-1}) \label{eq:A1} \\
  (\A-1) (1+r) &= 2d_+r(\kappa_+-1) - 2d_-(1-\kappa_-^{-1})\label{eq:A2}\\
  \A & = 1+d_-(\kappa_--2+\kappa_-^{-1})\label{eq:A3}.
  \end{align}
\end{subequations}
Then \eqref{eq:A1} leads to
\begin{equation}
 \kappa_+ ^2 - \left( 2 + \frac{\A-1}{d_+} \right)  \kappa_+ + 1 = 0 .
\end{equation}

In the region $|\A| > 1$, there are two roots, and their product is $1$; thus,
one of the two roots must meet $|\kappa_+| > 1$ and the other must meet
$|\kappa_+| <1$. We are interested only in the latter, so that the far-field
boundary condition is satisfied. Thus, we pick the solution
\begin{equation}
\kappa_+ = 1+ \frac{\A-1}{2d_+}\left(1-\sqrt{1+\frac{4d_+}{\A-1}} \right).
\end{equation}

Similarly for \eqref{eq:A3} we choose the solution
\begin{equation}
\kappa_-^{-1} = 1+ \frac{\A-1}{2d_-}\left(1-\sqrt{1+\frac{4d_-}{\A-1}} \right).
\end{equation}
Plugging these into \eqref{eq:A2} and simplifying the expression, we obtain
\begin{equation}
1 +r = \left( 1-\sqrt{1+ \frac{4d_-}{\A-1}} \right) + r \left( 1- \sqrt{1+ \frac{4d_+}{\A-1}} \right) .
\end{equation}

When $\A$ is real, the radicals must be zero. And the requirement that $\A \leq
1$ leads to $0 < d_+,d_- \leq \frac12$, which agrees with the classical CFL
condition that can be obtained by applying von Neumann stability analysis to an
uncoupled system. If $\A$ is complex, the imaginary parts of the two terms on
the right-hand side must cancel out, and their real parts must be zero. Thus the
radicals must be real, which contradicts with the assumption. Hence, there is no
stability loss in the coupling, and stability is guaranteed when classical CFL
conditions are satisfied on each domain regardless of the value of $r$.

\subsection{Explicit flux coupling with implicit time stepping}

Solving the interior equations implicitly and updating the interface data
explicitly, we have the following algorithm:
\begin{subequations}
  \begin{align}
  & T_{0_+}^{n+1} = T_{0_-}^n  \label{eq:implicit_css1}\\
  & T_j^{n+1}-T_j^{n} = \frac{\Delta t \, \nu_+}{\rho_+ c_+ (\Delta z_+)^2}  \left( T_{j+1}^{n+1} -2 T_j^{n+1} + T_{j-1}^{n+1}\right), \quad j>0 \label{eq:implicit_css2} \\
  &  T_{0_-}^{n+1}-T_{0_-}^{n} = \frac{2 \Delta t}{\rho_- c_- \Delta z_- + \rho_+ c_+ \Delta z_+}  \left(\nu_+  \frac{T_1^n - T_{0_+}^n}{\Delta z_+} - \nu_-  \frac{T_{0_-}^{n+1} - T_{-1}^{n+1} }{\Delta z_-} \right) \label{eq:implicit_css4} \\
  & T_j^{n+1}-T_j^{n} = \frac{ \Delta t \nu_-}{\rho_- c_- (\Delta z_-)^2}  \left( T_{j+1}^{n+1} -2 T_j^{n+1} + T_{j-1}^{n+1} \right), \quad j<0. \label{eq:implicit_css5}
  \end{align}
  \label{eq:implicit_css}
\end{subequations}

Assume a solution is of the form
\begin{equation*}
 T_j^n = \begin{cases} \A^n \kappa_-^j, & j = 0_-, -1, -2, \ldots \\
\A^{n-1} \kappa_+^j, & j = 0_+, 1,2, \ldots \end{cases}
\end{equation*}
so that \eqref{eq:implicit_css1} is automatically satisfied by this choice of
normal mode. The other three equations require that $\A, \kappa_-,$ and
$\kappa_+$ satisfy
\begin{align}
\begin{split}
1 & = \A^{-1} + d_+(\kappa_+ - 2 + \kappa_+^{-1}) \\
(1 - \A^{-1})(1+r) & =  2 d_+r \A^{-2}(\kappa_+-1) - 2d_- (1-\kappa_-^{-1}) \\
1 & = \A^{-1} + d_-(\kappa_- - 2 + \kappa_-^{-1}). \\
\end{split}
\label{eq:sim_implicit_css}
\end{align}

The first equation leads to
\begin{equation}
\kappa_+ ^2 - \left( 2 + \frac{1-\A^{-1}}{d_+} \right) \kappa_+ + 1 = 0 .
\end{equation}
In the region $|\A|>1$, there are two real roots with one of them being
$|\kappa_+|>1$ and the other being $|\kappa_+|<1$. Again we choose the latter,
which can be written as
\begin{equation}
 \kappa_+ = 1+ \frac{1-\A^{-1}}{2d_+}\left( 1  - \sqrt{1+ \frac{4d_+}{1-\A^{-1}}}\right).
\end{equation}

The third equation gives
\begin{equation}
\kappa_-^{-1} = 1+ \frac{1-\A^{-1}}{2d_-}\left( 1  - \sqrt{1+ \frac{4d_-}{1-\A^{-1}}}\right).
\end{equation}
Plugging these expressions into the second equation in
\eqref{eq:sim_implicit_css}
\iflongversion
gives
\begin{equation}
(1-\A^{-1})(1+r) = 2d_+r \A^{-2} \left(\frac{1-\A^{-1}}{2d_+}\left( 1  - \sqrt{1+ \frac{4d_+}{1-\A^{-1}}}\right) \right) - 2d_- \left(-  \frac{1-\A^{-1}}{2d_-}\left( 1  - \sqrt{1+ \frac{4d_-}{1-\A^{-1}}}\right) \right),
\end{equation}
which simplifies to
\fi
\begin{equation}
1 + r = r  \A^{-2} \left(1-\sqrt{1+\frac{4d_+}{1-\A^{-1}}} \right) +   \left(1-\sqrt{1+\frac{4d_-}{1-\A^{-1}}} \right) .
\label{eq:beuler_stab}
\end{equation}

Again, we consider asymptotic solutions.
\begin{itemize}
\item If $\Delta z_+ \ll \Delta z_-$, then $r \rightarrow 0$. In this case, we
are left with  $\sqrt{1+ \frac{4d_-}{1-\A^{-1}}} \approx 0$. Solving for $\A$
yields  $|\A| \approx \frac1{4d_-+1}< 1$. It follows that the scheme is
unconditionally stable.

\item If $\Delta z_- \ll \Delta z_+$, then $r \rightarrow \infty$. In this case, we have
\[
1 \approx \A^{-2} - \A^{-2}\sqrt{1 + \frac{4d_+}{1-\A^{-1}}}.
\]

After some algebra, it follows that $\A(-\A^4 + \A^3 + 2\A^2 - 2\A + 4d_+)
\approx 0$. Consider a special case $d_+ = 1$. We can easily find that all four
nonzero roots have a magnitude larger than one. Thus, the scheme is not
unconditionally stable. This is an interesting conclusion in an extreme scenario
and will be further examined in the next section.



\end{itemize}

\subsection{Discussion and comparison with Giles' results}
Giles considered a slightly different partitioned approach in \cite{Giles1997}.
In his approach, \eqref{eq:css4} is replaced with
\begin{equation}
  \frac12 \rho_- c_- \Delta z_- \frac{\partial}{\partial t} T_{0_-} = q_0 - q_{-\frac12},  \quad j=0, \label{eq:giles}
\end{equation}
where $q_0$ is computed in the same way as $q_\frac12$ in \eqref{eq:css3} but
deemed a one-sided approximation of the flux at the interface. The formula is
derived by using finite volume scheme over the half-sized cell
$[z_{-\frac12},0]$. By comparing \eqref{eq:giles} with \eqref{eq:css4}, we can
see that the only difference is the omission of $\rho_+ c_+ \Delta z_+$. The
impact on stability seems to depend on how $\Delta z_+$ compares with $\Delta
z_-$. Thus, Giles related the coupling stability to the parameter $r$ in
\eqref{eq:def} and investigated three algorithms: an explicit algorithm, an
implicit algorithm, and a hybrid algorithm (treating one domain explicitly and
the other domain implicitly). For all the algorithms, the following conclusions
were drawn:
\begin{itemize}
  \item When $r \ll 1$, the coupling is stable.
  \item When $r \gg 1$, the coupling is unstable.
  \iflongversion
  \item The stability condition is
  \begin{equation}
    r < \frac{1-2d_-}{1-\sqrt{1-2}},
  \end{equation}
  where $d$ is the Courant number $v\Delta t /\Delta z^2$. Note that this is
  established under the assumption that Dirichlet boundary condition is used on
  the positive domain (usually represents fluid) and the Neumann boundary
  condition  is used on the other (usually representing a solid). In practice,
  the fluid computation has higher resolution than the solid computation has,
  leading to a stable coupling. But if one uses the Dirichlet boundary condition
  for the solid and the Neumann boundary condition for the fluid, $r$ is the
  inverse of the original quantity and have a very large value. Then the
  coupling is not stable unless extremely small stepsize is used.
  \item The stability can be improved by using backward Euler for $j\le0$.
  The stability condition would be
  \begin{equation}
  r < \frac{1+2d_-}{1-\sqrt{1-2d_+}}.
  \end{equation}
  \item If both domains are treated implicitly but the interface data is
  explicitly updated (e.g., using time lagged information to promote
  parallelism), the asymptotic stability condition for $d_-,d_+ \gg 1$ becomes
  \begin{equation}
  r < \sqrt{\frac{d_-}{d_+}}.
  \end{equation}
  \fi
 \end{itemize}

In contrast, we consider contributions from both domains for the interface node.
The analysis of the explicit algorithm with explicit flux coupling shows that
the coupling stability is not affected by $r$ in our setup for the interface
node. For the implicit algorithm \eqref{eq:css}, our analysis leads to the same
conclusions as Giles predicted; however, the causes of instability in the two
algorithms are different. In Giles' algorithm, the contribution from the
dominant domain (positive domain if $r \gg 1$) is not fully accounted for at the
interface since the term $\rho_+ c_+ \Delta z_+$ is missing. With the influence
of this missing term excluded, it is easier to see in our framework that the
explicit coupling flux passed to the negative domain as Neumann boundary
condition dominates the right-hand side of \eqref{eq:implicit_css4}, changing
the discretization of this equation from a hybrid scheme to an explicit scheme.
To summarize the new findings, the choice of Dirichlet condition or Neumann
condition on each domain is important in order to determine the coupling
stability when explicit updating of the interfacial flux is used in implicit
methods, but it plays no role for purely explicit algorithms.

\section{Normal mode stability analysis for bulk interface condition}\label{sec:bulk}

We have shown that even using an implicit time-stepping Dirichlet-Neumann
condition and explicit flux coupling can lead to instability, regardless of  how
the interface condition is imposed at the interface. In this section, we show
how the stability properties of the coupling methods are affected in several
ways where the interfacial flux \eqref{eq:fv_mon1} is computed when the bulk
interface condition is imposed. In order to relax the stability constraint due
to explicit flux coupling, a natural choice is to add implicitness. Thus, we
consider the following algorithms based on the degree of implicitness used in
the flux computation:
\begin{enumerate}
    \item Explicit flux coupling
    \begin{equation}
    q_{0_+}^{n+1} = b (T_{\frac12}^n - T_{-\frac12}^n ).
    \end{equation}
    \item Partially implicit flux coupling
    \begin{equation}
    q_{0_+}^{n+1} = b (T_{\frac12}^{n+1} - T_{-\frac12}^n ).
    \end{equation}
    \item Implicit flux coupling
    \begin{equation}
    q_{0_+}^{n+1} = b (T_{\frac12}^{n+1} - T_{-\frac12}^{n+1} ).
    \end{equation}
\end{enumerate}

Explicit flux coupling and partial flux coupling are convenient for parallel
computing and easy to implement, requiring minimal modifications to existing
codes that support the Neumann boundary condition. Data transfer between the
coupling components is needed only at the beginning of a time step; the
frequency is determined by the larger of the stepsizes for each model. The fully
implicit treatment of the interfacial flux requires the solution of a monolithic
system, which normally indicates more data exchange and synchronization and,
more important, tremendous difficulties in developing partitioned algorithms
that can solve the equations in both domains simultaneously. However, a
sequentially implicit formulation (e.g., \cite{Farhat2006}) allows the
monolithic system to be solved in a partitioned manner (as described by the CSS
algorithm \eqref{eq:css}), while maintaining the same stability as a fully
implicit method \cite{Farhat2006}. This strategy has been widely used for
loosely coupled problems and has a rich literature (e.g.,
\cite{Causin2005,Bazilevs2008,Degroote2009}). Since the stability for the fully implicit case is obvious, the analysis will not be repeated in this paper.

\subsection{Explicit flux coupling with implicit time stepping}\label{sec:exfluximts}
Applying backward Euler to the equation \eqref{eq:css_bulk} gives
\begin{subequations}
  \begin{align}
  & q_{0_+}^{n+1} = b (T_{\frac12}^n - T_{-\frac12}^n ) \label{eq:implicit_bulk1} \\
  & T_{j+\frac12}^{n+1} - T_{j+\frac12}^n = \frac{\Delta t}{\rho_+ c_+ \Delta z_+}  \left( q_{j+1}^{n+1} - q_{j}^{n+1} \right), \quad j \geq 0 \label{eq:implicit_bulk2} \\
  & q_{0_-}^{n+1} = q_{0_+}^n \label{eq:implicit_bulk3} \\
  & T_{j+\frac12}^{n+1} - T_{j+\frac12}^n = \frac{\Delta t}{\rho_- c_- \Delta z_-}  \left(  q_{j+1}^{n+1} - q_{j}^{n+1}  \right), \quad j<0 .\label{eq:implicit_bulk4}
  \end{align}
  \label{eq:implicit_bulk}
\end{subequations}

\iflongversion
The two equations at the interface are
\begin{subequations}
  \begin{align}
  & T_{\frac12}^{n+1} - T_{\frac12}^n = \frac{\Delta t}{\rho_+ c_+ \Delta z_+}  \left( \nu_+ \frac{T_\frac32^{n+1} - T_\frac12^{n+1}}{\Delta z_+} - b (T_\frac12^n -T_{-\frac12}^n)  \right) \\
  & T_{-\frac12}^{n+1} - T_{-\frac12}^n = \frac{\Delta t}{\rho_- c_- \Delta z_-}  \left( b (T_\frac12^n -T_{-\frac12}^n) - \nu_-\frac{T_{-\frac12}^{n+1}-T_{-\frac32}^{n+1}}{\Delta z_-}  \right)
  \end{align}
\end{subequations}
\fi

We consider the normal mode solution
\begin{equation}
 T_{j+\frac12}^n = \begin{cases} \A^n \kappa_-^j, & j = -1, -2, \ldots \\
\A^n \kappa_+^{j+1}, & j = 0,1,2, \ldots . \end{cases}
\label{eq:bulk_normal_mode_solution}
\end{equation}
Substituting it into Equation \eqref{eq:implicit_bulk}, we have
\begin{align}
\begin{split}
1 - \A^{-1}  & = d_+(\kappa_+ - 2 + \kappa_+^{-1}) \\
1 - \A^{-1} & =  d_+(\kappa_+ - 1) - d_+ \Delta z_+ b/\nu_+ \A^{-1} \left( 1- \kappa_-^{-1} \kappa_+^{-1}\right) \\
1 - \A^{-1} & =  d_- \Delta z_- b/\nu_- \A^{-1} \left( \kappa_+ \kappa_- - 1\right) - d_-(1- \kappa_-^{-1}) \\
1 - \A^{-1}  & = d_-(\kappa_- - 2 + \kappa_-^{-1}), \\
\end{split}
\label{eq:sim_implicit_bulk}
\end{align}
where the second and the fourth equations correspond to the cases $j=0$ and
$j=-1$ in \eqref{eq:implicit_bulk}, respectively.
Solving the first and the last equations for $\kappa_+$ and $\kappa_-$,
respectively, and choosing the proper roots as before, we get
\begin{equation}
\begin{aligned}
 \kappa_+ = 1 + s_+ - \sqrt{s_+^2 + 2 s_+} \\
 \kappa_-^{-1} = 1 + s_- - \sqrt{s_-^2 + 2 s_-},
\end{aligned}
\label{eq:sim_implicit_bulk_k}
\end{equation}
where
\begin{equation}
s_\pm = \frac{1 - \A^{-1}}{2 d_\pm}.
\end{equation}
The second and third equations of \eqref{eq:sim_implicit_bulk} lead to
\begin{equation}
  \A
  = \frac{1+ \beta_+ \left( 1 - \kappa_-^{-1} \kappa_+^{-1} \right)}{1- d_+ \left( \kappa_+ -1 \right)}
\label{eq:A_exflux_bulk2}
\end{equation}
and
\begin{equation}
  \A
  = \frac{1+ \beta_- \left( 1 - \kappa_- \kappa_+ \right)}{1- d_- \left( \kappa_-^{-1} -1 \right)} ,
  \label{eq:A_exflux_bulk3}
\end{equation}
respectively, where for notational convenience we define
\begin{equation}
  \beta_\pm = \frac{b \Delta t} {\rho_\pm c_\pm \Delta z_\pm}
  \label{def:bulk_courant}
\end{equation}
and refer to it as {\em bulk Courant number}, considering its similarity to
classical Courant numbers.

Plugging \eqref{eq:sim_implicit_bulk_k} into \eqref{eq:A_exflux_bulk2} and
\eqref{eq:A_exflux_bulk3} would result in equations that take $\A$ as functions
of $\beta_\pm$ and $d_\pm$, but there are no closed-form solutions. Thus we
consider the asymptotic behavior when $\beta_+$ approaches zero. Equation
\eqref{eq:A_exflux_bulk2} becomes
\begin{equation}
  \A
  = \frac{1}{1- d_+ \left( \kappa_+ -1 \right)}.
\label{eq:A_exflux_bulk2_asym}
\end{equation}
Plugging \eqref{eq:sim_implicit_bulk_k} into it and simplifying lead to
\begin{equation}
\A d_+ (\A - 1) = 0.
\end{equation}
Therefore $\A = 1$, which implies that the system is stable. Similarly with
\eqref{eq:A_exflux_bulk3}, we can see the system is stable when $\beta_-$
approaches zero.

Deriving general stability constraints is difficult since $\A$ is a complicated
function of $d_+$, $d_-$, and $\beta_+$ in \eqref{eq:A_exflux_bulk2} or a
function of $d_-$, $d_+$, and $\beta_-$ in \eqref{eq:A_exflux_bulk3}.
Nevertheless, if we omit the feedback from the coupling domain, that is,
$\kappa_-$ in \eqref{eq:A_exflux_bulk2} or $\kappa_+$ in
\eqref{eq:A_exflux_bulk3}, a nice closed-form expression for stability
constraint can be found in this best case. Therefore, in the general case the
coupling is not absolutely stable. Details are given in Section
\ref{sec:forced_model}.


\subsection{Partially implicit flux coupling with implicit time stepping}\label{sec:imfluximts}
Like the explicit case, the partially implicit flux computation allows the two
components to be handled simultaneously. But the equation corresponding to the
interface node needs an implicit solve. The full algorithm is
\begin{subequations}
  \begin{align}
  & q_{0_+}^{n+1} = b (T_{\frac12}^{n+1} - T_{-\frac12}^n ) \\
  & T_{j+\frac12}^{n+1} - T_{j+\frac12}^n = \frac{\Delta t}{\rho_+ c_+ \Delta z_+}  \left( q_{j+1}^{n+1} - q_{j}^{n+1} \right), \quad j \geq 0 \\
  & q_{0_-}^{n+1} = b (T_{\frac12}^{n} - T_{-\frac12}^{n+1} ) \\
  & T_{j+\frac12}^{n+1} - T_{j+\frac12}^n = \frac{\Delta t}{\rho_- c_- \Delta z_-}  \left(  q_{j+1}^{n+1} - q_{j}^{n+1}  \right), \quad j<0 .
  \end{align}
  \label{eq:pimflux_bulk}
\end{subequations}
Inserting the normal mode solution \eqref{eq:bulk_normal_mode_solution} into
Equation \eqref{eq:pimflux_bulk}, we have
\begin{align}
\begin{split}
1 - \A^{-1}  & = d_+(\kappa_+ - 2 + \kappa_+^{-1}) \\
1 - \A^{-1} & =  d_+(\kappa_+ - 1) - \beta_+ \left( 1- \A^{-1} \kappa_-^{-1} \kappa_+^{-1}\right) \\
1 - \A^{-1} & = \beta_- \left( \A^{-1} \kappa_- \kappa_+ - 1\right) - d_-(1- \kappa_-^{-1}) \\
1 - \A^{-1}  & = d_-(\kappa_- - 2 + \kappa_-^{-1}) .\\
\end{split}
\label{eq:sim_pimflux_bulk2}
\end{align}
The second and third equations give
\begin{equation}
  \A = \frac{1 + \beta_+ \kappa_-^{-1} \kappa_+^{-1}}{\beta_+ + 1 - d_+ \left( \kappa_+ - 1\right)}
\label{eq:A_pimflux_bulk2}
\end{equation}
and
\begin{equation}
  \A = \frac{1 + \beta_- \kappa_- \kappa_+}{\beta_- + 1 - d_- \left( \kappa_-^{-1} - 1 \right)},
\label{eq:A_pimflux_bulk3}
\end{equation}
respectively. We distinguish between two situations: $| \kappa_-^{-1}
\kappa_+^{-1}| \leq 1$ or $|\kappa_- \kappa_+| \leq 1$. For the former case, we
can observe that
\[
  | 1 + \beta_+ \kappa_-^{-1} \kappa_+^{-1} | \leq 1 + \beta_+
\]
and
\[
  | \beta_+ + 1 - d_+ \left( \kappa_+ - 1\right) | > 1 + \beta_+.
\]
Thus we have $|\A| < 1$. This is also true for the latter case. Therefore, the
system is unconditionally stable.


\section{Analysis for one-way coupled model}\label{sec:forced_model}
In this section, we present a stability analysis of a one-way coupled diffusion
model, which is forced by the other component in the coupled system but does not
provide feedback to the forcing component \cite{Beljaars2017}. In other words,
one of the domains treats the variables from the other domain as a boundary
condition. The advantage of using this simpler system is to enable explicit
derivation of the stability criterion, and therefore give insight into the
coupled stability. Again, we analyze the changes to stability induced by
different treatments of the interfacial flux.
\subsection{Stability analysis for the diffusion equation}

Let us consider the forced diffusion equation solely on the negative domain, as
is expressed by
\iflongversion
\begin{subequations}
\begin{align}
\frac{\partial}{\partial t} T_- = \frac{\partial}{\partial z} (K_- \frac{\partial}{\partial z} T_-), \quad z < 0 \\
K_- \frac{\partial}{\partial z} T_- = b (T_+ - T_-), \quad z =0 .
\label{eq:1ddiff_sn}
\end{align}
\end{subequations}
\fi
Explicit coupling boundary yields
\iflongversion
\begin{subequations}
\begin{align}
  \label{eq:1ddiff_sn_exflux_a}
  T_{j+\frac12}^{n+1} - T_{j+\frac12}^n = \frac{\Delta t}{\rho_- c _- \Delta z_-}  \left(  q_{j+1}^{n+1} - q_{j}^{n+1}  \right), \quad j< 0 \\
  \label{eq:1ddiff_sn_exflux_b}
  q_{j}^{n+1} = \nu_- \frac{T_{j+\frac12}^{n+1} - T_{j-\frac12}^{n+1} }{\Delta z_-}, \quad q_{0}^{n+1} = b (T_{\frac12}^{n} - T_{-\frac12}^{n}) .
\end{align}
\label{eq:1ddiff_sn_exflux}
\end{subequations}
\fi
In the simplest case of ocean alone, for stability analysis, we can have a fixed
boundary that is set to $T_{\frac12}^{n}=0$ so that
\begin{equation}
q_{0}^{n+1} = - b T_{-\frac12}^{n}.
\end{equation}
Then the method for solving the diffusion equation can be written as
\begin{subequations}
\begin{align}
  T_{j+\frac12}^{n+1} - T_{j+\frac12}^n = d_-  \left( T_{j+\frac32}^{n+1} - 2 T_{j+\frac12}^{n+1} + T_{j-\frac12}^{n+1}   \right), \quad j< -1 \\
  T_{-\frac12}^{n+1} - T_{-\frac12}^n = -\beta_- T_{-\frac12}^{n} - d_- \left(T_{-\frac12}^{n+1} - T_{j-\frac32}^{n+1} \right), \quad j=-1 .
\end{align}
\label{eq:1ddiff_sn_dis}
\end{subequations}

If there is no boundary, we can consider the solution
\begin{equation}
T_{j+\frac12}^{n+1} = \A^{n+1} e^{i(j m \Delta z_-)}\,,
\label{eq:sol_nobound}
\end{equation}
so that
\begin{align*}
& \A^{n+1}e^{i m j \Delta z_-}  (1-\A^{-1} )= d_- \A^{n+1} e^{i m j \Delta z_-}  (e^{i m \Delta z_-} -2 + e^{-i m\Delta z_-}), \quad j< -1 \\
\Rightarrow & \quad 1 - \A^{-1}= d_- (e^{i m \Delta z_-} -2 +e^{-i m \Delta z_-} ) = 2 d_- ( cos(m\Delta z_-) - 1) \\
\Rightarrow & \quad A^{-1}=1+2 d_- (1-cos(m \Delta z_-)) > 1 \,,
\end{align*}
which results in an unconditionally stable scheme that is in agreement with the
classical theory. 

With the boundary condition, we establish the following theorem for the
stability condition.
\begin{theorem}
Method \eqref{eq:1ddiff_sn_dis} applied to the diffusion model
\eqref{eq:1dadvdiff2} yields a stable solution if and only if $\beta_- \leq 1 +
\sqrt{1+2d_-}$.
\end{theorem}
\begin{proof}
We can no longer assume a solution of form \eqref{eq:sol_nobound}. Instead, we
assume the eigen-solution of the form
\begin{equation}
T_{j+\frac12}^{n+1} = \A^{n+1} \kappa_-^j .
\label{eq:sol_bound}
\end{equation}
A stable solution that satisfies the far-field boundary condition (at
$j\rightarrow -\infty$, $\kappa_-^j \rightarrow 0$) is $|\kappa_-|>1$ and $|\A|
\leq 1$.
From \eqref{eq:1ddiff_sn_dis} we have
\begin{subequations}
\begin{align}
1 - \A^{-1}  & = d_-(\kappa_- - 2 + \kappa_-^{-1}) \label{eq:1ddiff_sn_sim1}\\
1 - \A^{-1} & = -\beta_- \A^{-1} - d_-(1- \kappa_-^{-1})\,, \label{eq:1ddiff_sn_sim2}
\end{align}
\end{subequations}
which gives
%
\begin{equation}
    \kappa_-^{-1} = \frac{d_- + 1 + (\beta_- -1)\A^{-1}} {d_-} .
    \label{eq:1ddiff_sn_kappa}
\end{equation}

Substituting into  \eqref{eq:1ddiff_sn_sim1}, we have an equation for $\A$,
\begin{equation}
d_- \A^2-(\beta_- + d_-)\A- \beta_-(\beta_- -1)=0 ,
\end{equation}
with roots
\begin{equation}
\A = \frac{(\beta_-+d_-) \pm \sqrt{(\beta_- + d_-)^2 + 4 \beta_-(\beta_- -1) d_-}}{2 d_-}\,,
\label{eq:1ddiff_sn_aroots}
\end{equation}
which are both real because under the radical is a positive real number.

Now we show that choosing the positive sign does not lead to a stable solution.
First, we can observe from \eqref{eq:1ddiff_sn_kappa} that $\kappa_-$ is also
real. When choosing the positive sign, we have $0< \A \leq 1$, and thus
$1-\A^{-1} \leq 0$. According to \eqref{eq:1ddiff_sn_kappa}, the right-hand side
must be no larger than zero, which means $k_- \leq 0$. Equation
\eqref{eq:1ddiff_sn_kappa} indicates that $k_-$ can  be negative or zero only
when $\beta_- \leq 1$.

With $\beta_- \leq 1$ and \eqref{eq:1ddiff_sn_aroots}, we obtain
\begin{align}
  \nonumber
\A& = \frac{(\beta_-+d_-) + \sqrt{(\beta_- + d_-)^2 + 4
    \beta_-(\beta_- -1) d_-}}{2 d_-}\\
\nonumber
&\geq \frac{(\beta_-+d_-) + \sqrt{(\beta_- + d_-)^2}}{2 d_-} > 1+ \frac{\beta_-}{d_-} > 1,
\end{align}
which contradicts the assumption $|\A| \leq 1$.

Selecting the root with negative sign in \eqref{eq:1ddiff_sn_aroots}, we
consider the following two cases:
\begin{enumerate}
  \item If $\beta \geq 1$, the condition $|\A|\leq 1 $ implies
  \begin{equation}
      \beta_-^2 - 2\beta_- -2 d_- \leq 0.
  \end{equation}
  Solving the quadratic inequality gives
  \begin{equation}
      \beta_- \leq 1 + \sqrt{1+2d_-}\,.
  \end{equation}
  \item If $\beta < 1$, we have a stable solution because
  \iflongversion
  \begin{align*}
  \A &= \frac{(\beta_-+d_-) - \sqrt{(\beta_- + d_-)^2 + 4 \beta_-(\beta_- -1) d_-}}{2 d_-} \\
  &< \frac{(\beta_-+d_-) - \sqrt{(\beta_- - d_-)^2}}{2 d_-}
  \leq \frac{(\beta_-+d_-) - |\beta_- - d_-|}{2 d_-} \\
  &\leq \frac{(\beta_-+d_-) - \beta_- + d_-}{2 d_-} = 1\,.
  \end{align*}
  \fi
\end{enumerate}
Combining these two cases, we can see that $\beta_- \leq 1 + \sqrt{1+2d_-}$ is
the final condition for stability.
\end{proof}

In Figure \ref{fig:exflux_stab} we illustrate the stability regions resulting
from our analysis. The norm of $A$ and $\kappa^{-1}$ are calculated for a range
of values of $\beta_-$ and $d_-$ according to the expressions
\eqref{eq:1ddiff_sn_aroots} and \eqref{eq:1ddiff_sn_kappa}. The boundary of the
stability regions agrees exactly with the relationship established by the
analysis.

Based on the same model problem, Beljaars et al. derived an empirical stability
boundary of $\beta_- \leq 2+ \sqrt{d_-}^{1.1}$ in \cite{Beljaars2017}, which is
close to our analytical result. So, if there is no mixing, the explicit scheme
is stable for weak coupling but unstable for strong coupling.

Even a decrease of stepsize $\Delta t$ for the negative domain is not effective
in suppressing coupling, because it reduces both $d_-$ and $\beta_-$. Instead,
the most effective way to suppress coupling instability is to reduce $\Delta
z_-$, which increases $d_-$ much more rapidly than $\beta_-$.
%
\begin{figure}
  \centering
  \includegraphics[width=0.8\linewidth]{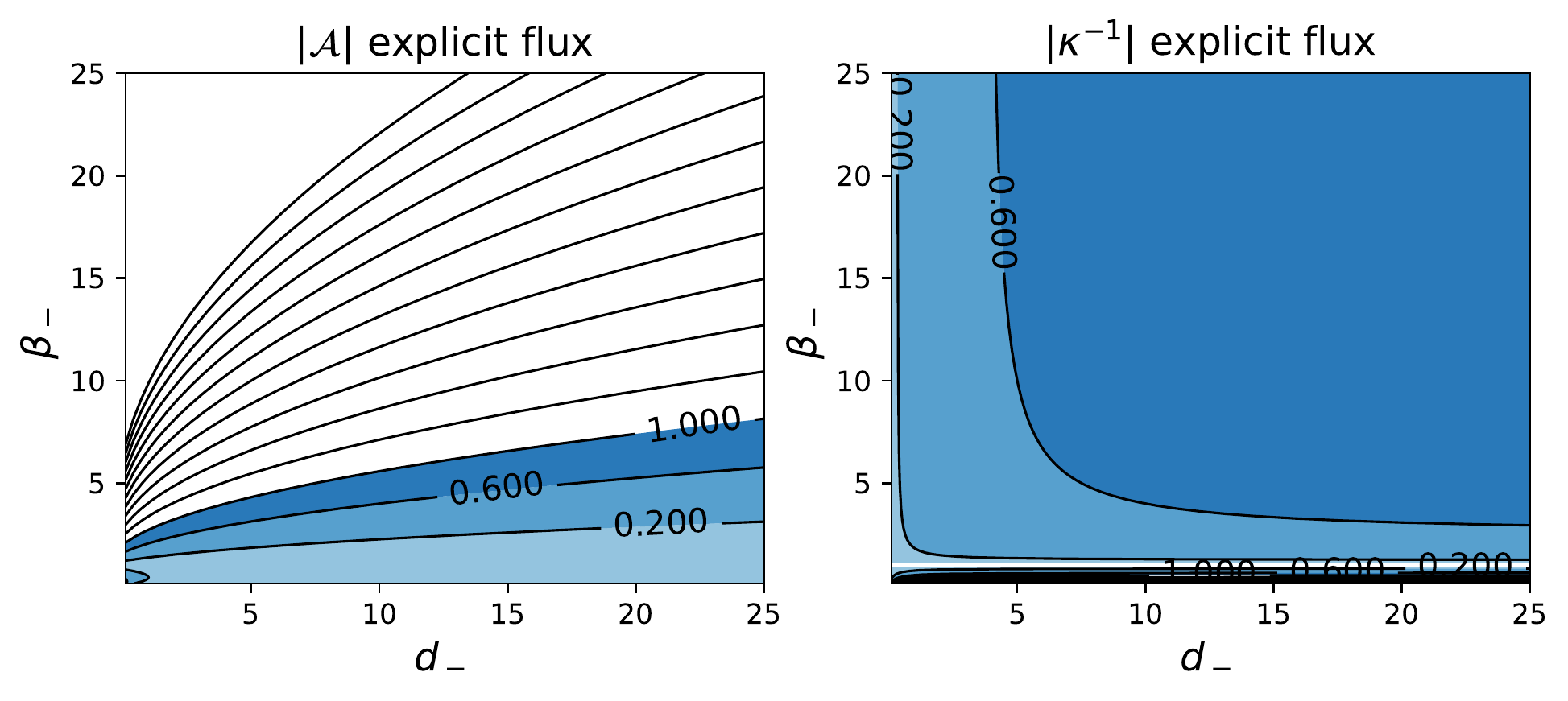}
  \caption{Stability regions for explicit flux coupling (based on analytical formula).}
  \label{fig:exflux_stab}
\end{figure}

\subsection{Implicit interfacial flux}
Next we consider the implicit treatment of the interfacial flux. The equation
for the boundary node $j=-1$ changes to
\begin{equation}
  T_{-\frac12}^{n+1} - T_{-\frac12}^n = -\beta_- T_{-\frac12}^{n+1} - d_- \left(T_{-\frac12}^{n+1} - T_{j-\frac32}^{n+1} \right), \quad j=-1.
  \label{eq:1ddiff_sn_dis2}
\end{equation}
\iflongversion
Plugging \eqref{eq:sol_bound} into \eqref{eq:1ddiff_sn_dis2}, we have
\begin{equation}
1 - \A^{-1} = -\beta_-  - d_-(1- \kappa_-^{-1}).
\label{eq:1ddiff_sn_sim3}
\end{equation}
Then we can solve \eqref{eq:1ddiff_sn_sim3} and \eqref{eq:1ddiff_sn_sim1} for
$\kappa_-^{-1}$.
\fi

After some manipulation, we obtain
\begin{equation}
\A = \frac{\beta_--d_-}{\beta_- - d_-+\beta^2}\,,\quad \kappa_-^{-1} = \frac{d_-}{d_--\beta_-} .
\label{eq:1ddiff_sn_imflux_A_kappa}
\end{equation}

\begin{theorem}
  The method \eqref{eq:1ddiff_sn_dis} with the equation for the interface node
  replaced by \eqref{eq:1ddiff_sn_dis2} gives a stable solution of the one-way
  coupled diffusion model \eqref{eq:1dadvdiff2}.
\end{theorem}
\begin{proof}

\iflongversion
\begin{enumerate}
  \item If $\beta_- \geq d_-$, $|\A|\leq 1$ always holds.
  \item If $\beta_- < d_-$, $|\A|\leq 1$ leads to
  \begin{equation}
    \beta_-^2+2 \beta_- - 2 d_- \geq 0
  \end{equation}
  which reduces to
  \begin{equation}
    \beta_- \geq \sqrt{1+2d_-} -1
  \end{equation}
\end{enumerate}
Combining these two cases, we have
\begin{equation}
    \beta_- \geq \sqrt{1+2d_-} -1
\end{equation}
\fi
For $|\kappa^{-1}|\leq 1$, we have
\begin{equation}
    \beta_- \geq 2d_-.
\end{equation}
Since $2d_- > \sqrt{1+2d_-} -1$, $|\A| \leq 1$ always holds for all the eigen
modes with $|\kappa^{-1}|\leq 1$, thus stability follows.
\end{proof}

Figure \ref{fig:imflux_stab_u0} graphically depicts the stability conditions for
the implicit case that are calculated directly using
\eqref{eq:1ddiff_sn_imflux_A_kappa}. It clearly shows that the region where
$|\kappa^{-1}|\leq 1$ is a subset of region where $|\A| \leq 1$.
\begin{figure}
  \centering
  \includegraphics[width=0.8\linewidth]{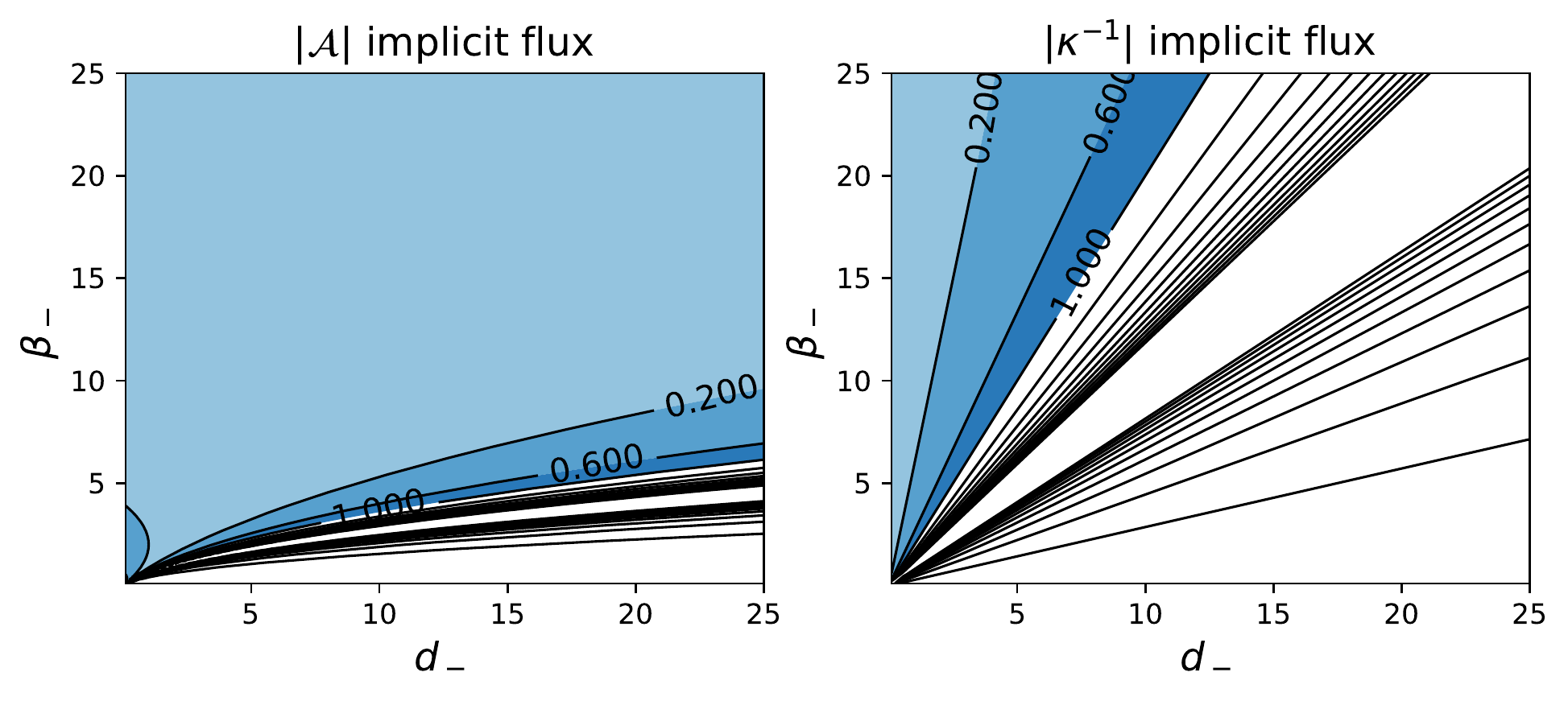}
  \caption{Stability regions for implicit flux coupling (based on analytical formula).}
  \label{fig:imflux_stab_u0}
\end{figure}

\section{Numerical validation via matrix stability analysis}\label{sec:validation}
While the normal mode analysis provides insights into the stability of the
coupled algorithms, we use a matrix stability analysis as a numerical tool to
validate the results.
To incorporate the influence of flux coupling in a simple way, we cast the
partitioned algorithms in a monolithic framework that consists of each component
and provides a unified view for different coupling approaches.


Depending on the flux coupling strategies, the interfacial fluxes $q_{\frac12}$
and $q_{-\frac12}$ can be computed in different ways. However, we can rewrite
the fully discretized equation for \eqref{eq:css} in a general matrix form
\begin{equation}
 \mathbf{A} \mathbf{T}^{n+1} = \mathbf{B} \mathbf{T}^{n}.
 \label{eq:matrix_form}
\end{equation}
Here $\mathbf{T}$ is represented as
\begin{equation}
  \quad \mathbf{T} = [ T_{N_-},\dots,T_{-2},T_{-1},T_0,T_{1},T_{2},\dots,T_{N_+}]^T
\end{equation}
for the Dirichlet-Neumann condition and
\begin{equation}
\quad \mathbf{T} = [ T_{N_-},\dots,T_{-2},T_{-1},T_{1},T_{2},\dots,T_{N_+}]^T
\end{equation}
for bulk condition due to the differences in the grid settings. In the
following, we present the results for the bulk interface condition and then the
Dirichlet-Neumann condition. Both the two-way coupled and the one-way coupled models are included 

\subsection{Bulk condition}
For the backward Euler scheme, the corresponding $\mathbf{A}$ and $\mathbf{B}$
are
\begin{subequations}
  \footnotesize
  \begin{align}
    &\mathbf{A} = \\
    &\nonumber\left[
    \resizebox{0.87\textwidth}{!}{$\displaystyle
    \begin{array}{c c c c c : c c c c c}
    2 d_-+1 & -d_-    &         &         &         &         &         &         &         &         \\
    -d_-    & 2 d_-+1 & -d_-    &         &         &         &         &         &         &         \\
            & \ddots  & \ddots  & \ddots  &         &         &         &         &         &         \\
            &         &    -d_- & 2 d_-+1 & -d_-    &         &         &         &         &         \\
            &         &         & -d_-    & d_-+\theta \beta_- +1& -\gamma \beta_- &         &         &         &       \\
             \hdashline
            &         &         &         & -\gamma \beta_+  & d_+ + \theta \beta_+ + 1& -d_+    &         &         &         \\
            &         &         &         &         & -d_+    & 2 d_++1 & -d_+    &         &         \\
            &         &         &         &         &         & \ddots  & \ddots  & \ddots  &         \\
            &         &         &         &         &         &         &    -d_+ & 2 d_++1 & -d_+    \\
            &         &         &         &         &         &         &         & -d_+    & 2 d_++1
    \end{array}
    $}
    \right] \\
    &\mathbf{B} =
        \left[
    \resizebox{0.6\textwidth}{!}{$\displaystyle
    \begin{array}{c c c c c : c c c c c}
    1 &   &        &   &                       &                       &   &        &   &     \\
      & 1 &        &   &                       &                       &   &        &   &     \\
      &   & \ddots &   &                       &                       &   &        &   &     \\
      &   &        & 1 &                       &                       &   &        &   &     \\
      &   &        &   & 1-(1-\theta)\beta_- & (1-\gamma)\beta_-   &   &        &   &     \\
           \hdashline
      &   &        &   & (1-\gamma)\beta_+    & 1-(1-\theta)\beta_+ &   &        &   &     \\
      &   &        &   &                       &                       & 1 &        &   &     \\
      &   &        &   &                       &                       &   & \ddots &   &     \\
      &   &        &   &                       &                       &   &        & 1 &     \\
      &   &        &   &                       &                       &   &        &   & 1
    \end{array}
    $}
    \right] ,
  \end{align}
  \label{eq:matrix_form_bulk}
\end{subequations}
where $\theta \in \{0,1\}$ and $\gamma \in \{0,1\}$ defines different coupling methods:
\begin{enumerate}
    \item Explicit flux coupling $(\theta,\gamma) = (0,0) $.
    \item Partially implicit flux coupling $(\theta,\gamma) = (1,0) $.
    \item Implicit flux coupling  $(\theta,\gamma) = (1,1) $.
\end{enumerate}
Since we focus on practical partitioned algorithms, we
consider the sequentially implicit formulation rather than the fully implicit
formulation. In this formulation, the upper right part of $\mathbf{A} $ will be
zero, the upper right part of $\mathbf{B} $ will be $\gamma_-$, and the rest will be
unchanged.

The one-way coupled model in Section \ref{sec:forced_model} can also be
represented in this form with $\mathbf{A}$ and $\mathbf{B}$ being only the upper
left blocks of the original matrices. $\theta=0$ and $\theta=1$ correspond to
the explicit case and partially implicit case, respectively.

The numerical scheme is considered  stable if all the eigenvalues of the
matrix $\mathbf{M} = \mathbf{A^{-1}}\mathbf{B}$ are within the unit circle.

\iflongversion
In the numerical simulations, we choose the parameter values that are commonly
used in climate models such as E3SM. A list of parameter values is given in
Table \ref{tab:params} for reference. Varying the variable $\beta$ and $d$ is
achieved by adjusting spatial and temporal resolutions and changing the depth in
$z$ direction. The same number of grid points is used for the ease of
computation.
\begin{table}
  \centering
\resizebox{0.9\textwidth}{!}{
\renewcommand{\arraystretch}{1.2}
\begin{tabular}{c | c c c c}
\hline
                            & Parameter    & Description           & Value         & Unit \\
\hline
\multirow{5}{*}{Atmosphere} & $\rho_+ $    & density               & $1$           & \si{kg\,m^{-3}}  \\
                            & $c_+$        & heat capacity         & $1000$        & \si{J\,kg^{-1}\,K^{-1}} \\
                            & $\nu_+$      & dynamic diffusivity & $[300,?]$           & \si{J\,s^{-1}\,m^{-1}\,K^{-1}}\\
                            & $K_+$      & eddy diffusivity& $[0.3,?]$           & \si{m^2\,s^{-1}}\\
                            & $n_+$ & grid points  &  $10$        &   \\
                            & $d_+$        & parabolic CFL         & $0.5$              &   \\
                            \hline
\multirow{5}{*}{Ocean}      & $\rho_- $    &  density              & $1000$        & \si{kg\,m^{-3}}  \\
                            & $c_-$        & heat capacity         & $4000$        & \si{J\,kg^{-1}\,K^{-1}} \\
                            & $\nu_-$      & diffusion coefficient & $[4\times10^{5},?]$           & \si{m^2\,s^{-1}}\\
                            & $n_-$ & grid points  &  $20$        &   \\

                            & $d_-$        & parabolic CFL         & $100$              &   \\
                            \hline
\multirow{1}{*}{}           & $b $      &  bulk coefficient        & $[5,100]$     & \si{J\,s^{-1}\,m^{-2}\,K^{-1}}  \\
\hline
\end{tabular}
} \caption{Typical parameters in practice. The smallest diffusivity coefficients
are of particular interest. This is because instability is more likely to appear
in these regime.
}
\label{tab:params}
\end{table}
\fi
%
%
\subsubsection{Two-way coupled model}

Since the stability of the coupled model is determined by the combination of the
 dynamics on both domains, first we look at the stability regions by taking
 $\lambda_{max} = \max(|\mathrm{eig}(\mathbf{M})|)$ as a function over the
 Courant numbers $\beta_-$ and $d_-$ for the negative domain while fixing
 $\beta_+$ and $d_+$. Figure \ref{fig:cpl_stab} shows that using explicit flux
 coupling does not lead to absolute stability, a result that agrees with the
 finding in Section \ref{sec:exfluximts} for the coupled system. And as
 expected, both partially implicit flux coupling and implicit flux coupling can
 make the coupling unconditionally stable.

To investigate contributions of the positive domain to stability, we vary
$\beta_+$ and $d_+$ by changing the grid size. As shown in Figure
\ref{fig:exflux_cpl_stab}, the stability regions for the explicit flux coupling
shrink as $\beta_+$ and $d_+$ increase.
%
\begin{figure}
  \centering
  \includegraphics[width=0.3\linewidth]{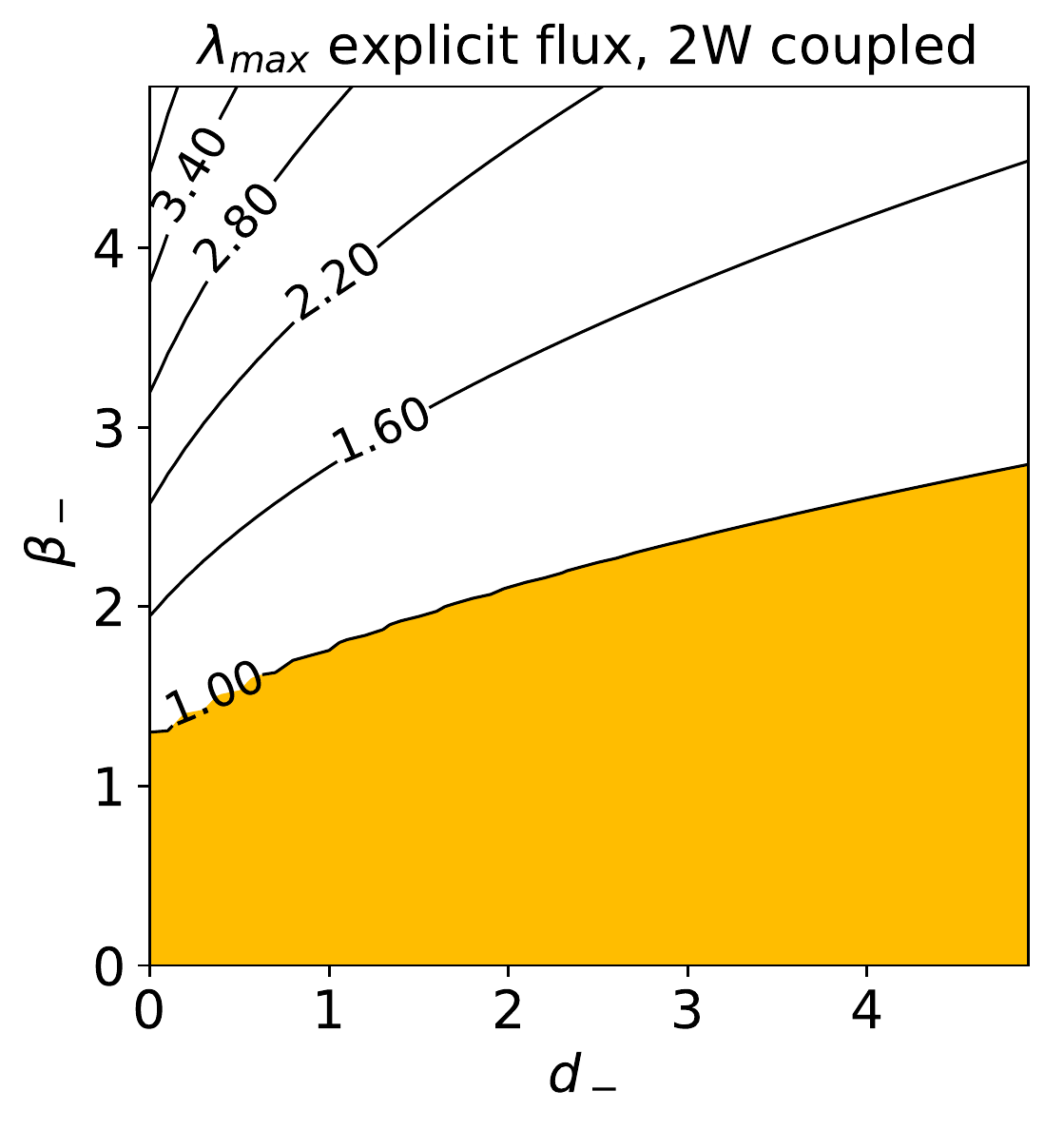}
  \includegraphics[width=0.33\linewidth]{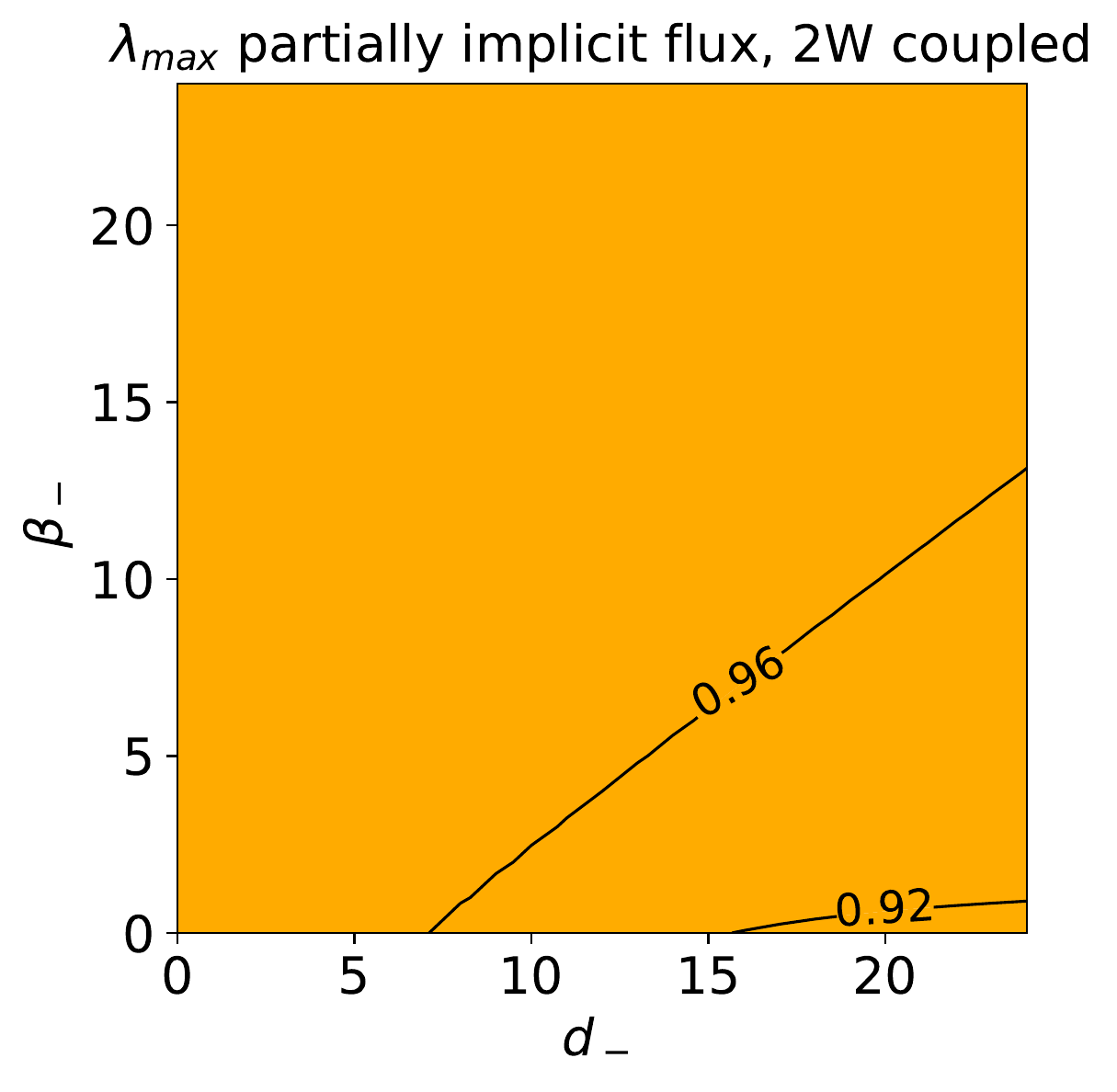}
  \includegraphics[width=0.308\linewidth]{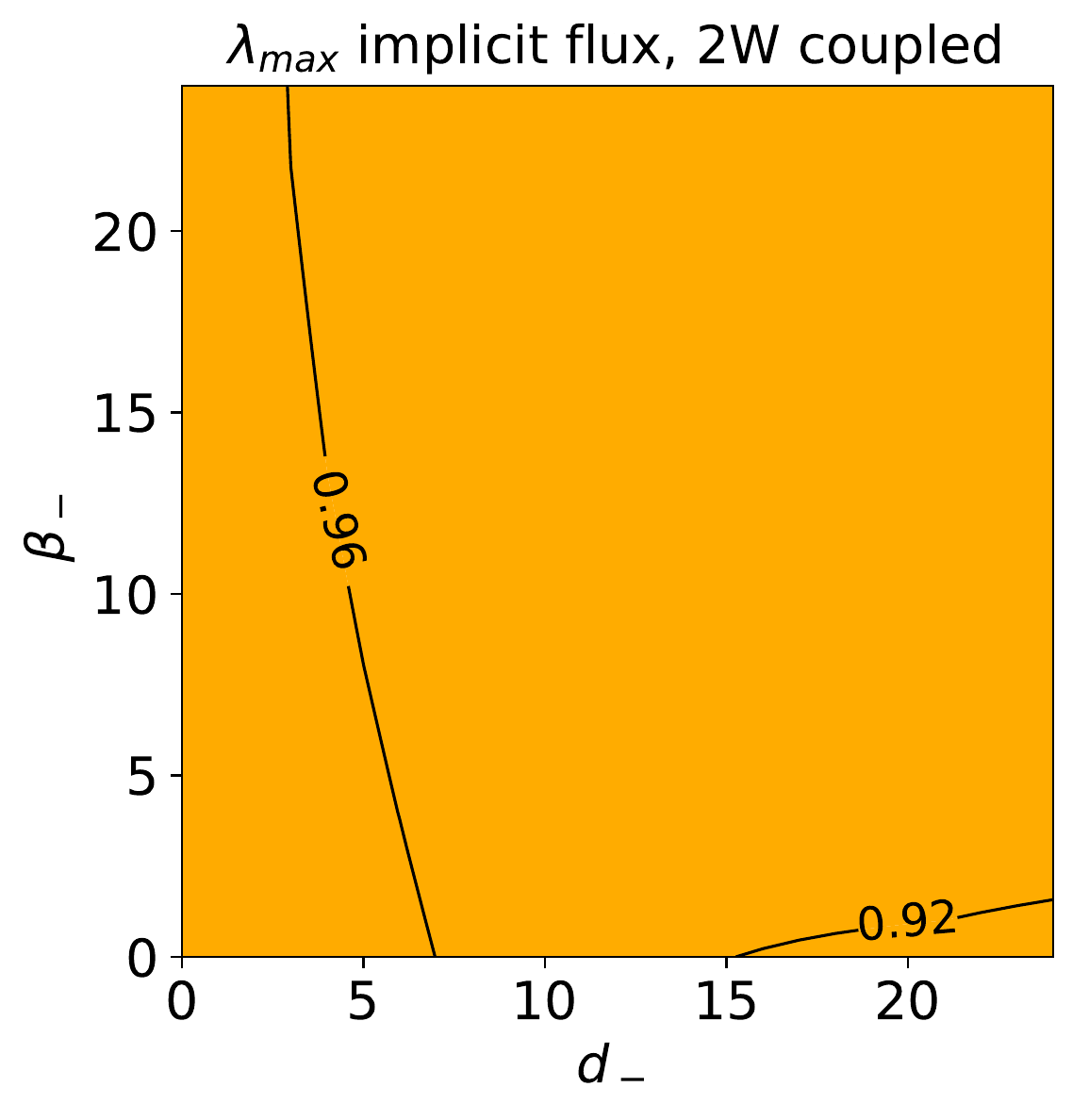}
  \caption{Stability regions for various flux coupling strategies with fixed
  parameters $\beta_+=1.125$ and $d_+=2.025$ for the coupling component.}
  \label{fig:cpl_stab}
\end{figure}
\begin{figure}
  \centering
  \includegraphics[width=0.32\linewidth]{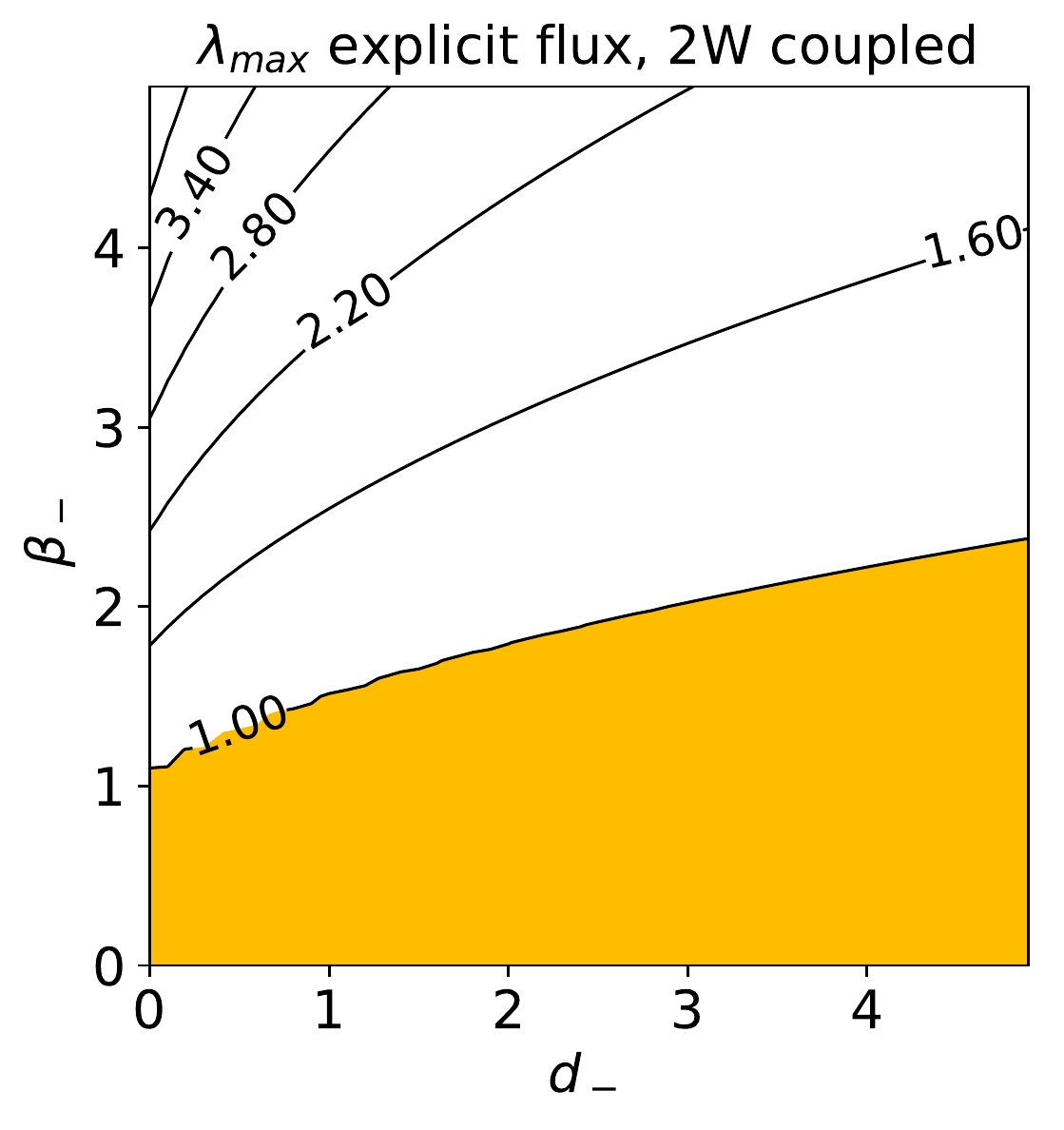}
  \includegraphics[width=0.32\linewidth]{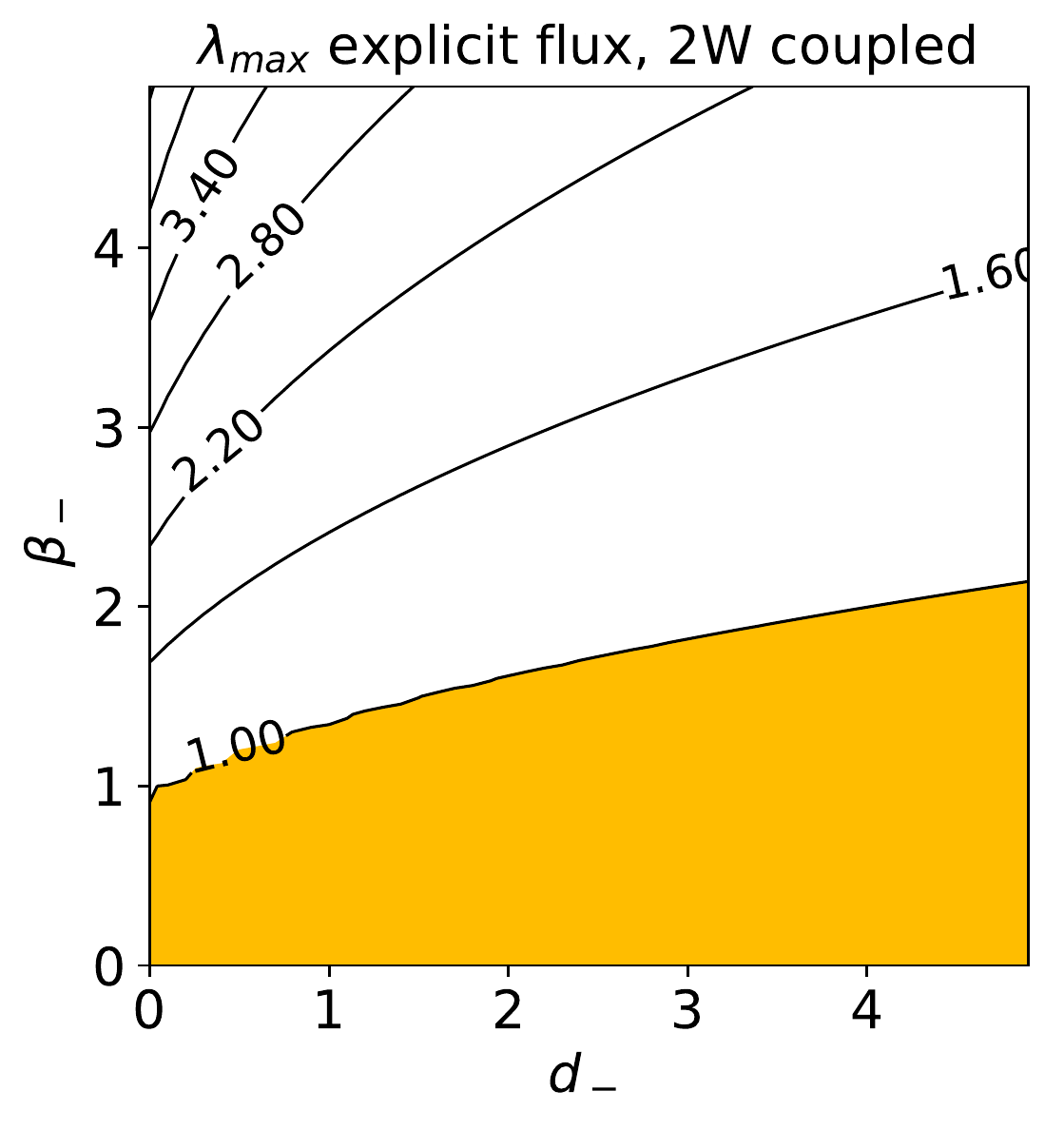}
  \includegraphics[width=0.32\linewidth]{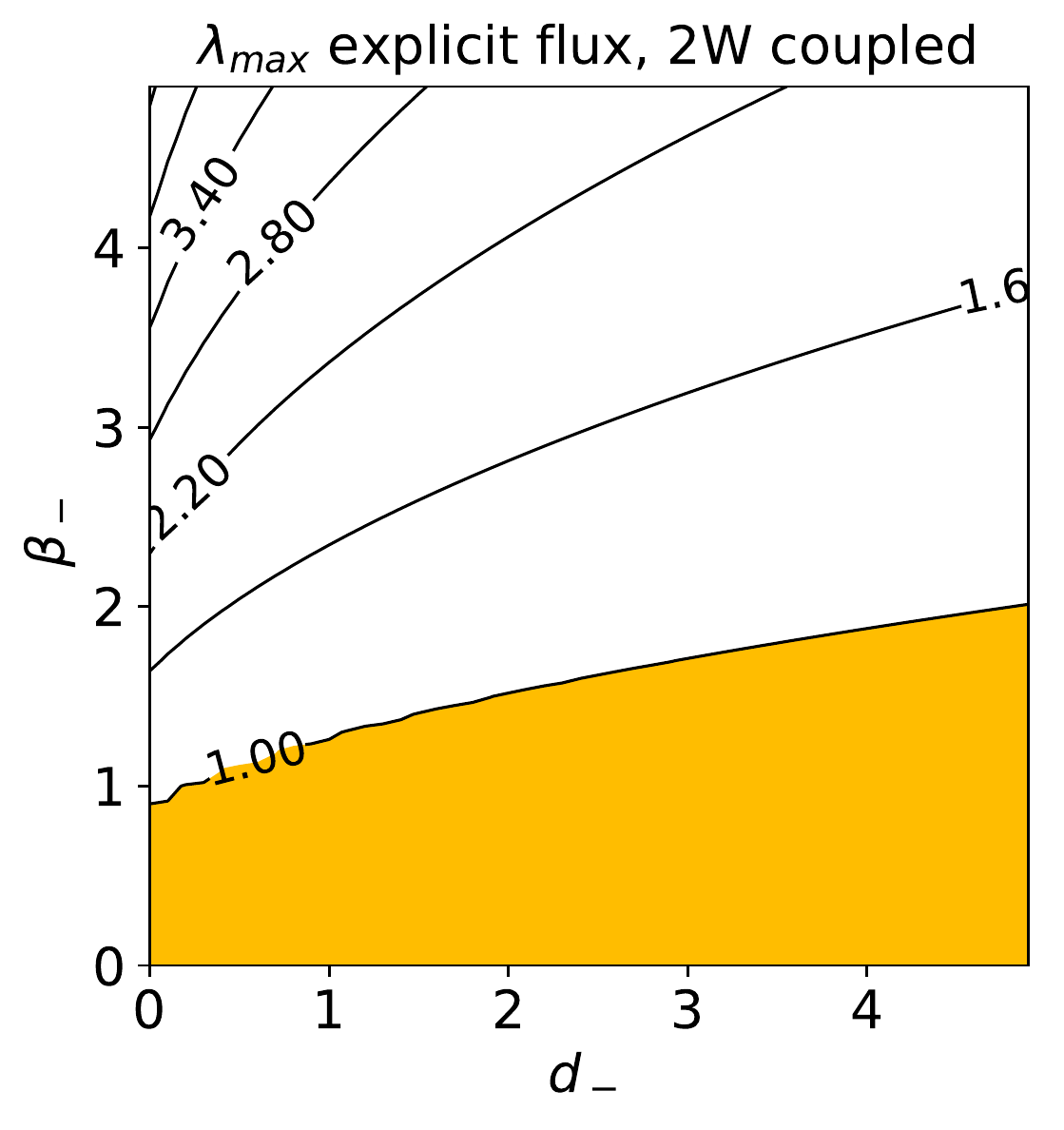}
  \caption{Stability regions for explicit flux coupling with fixed parameters $(\beta_+, d_+)$ set to $(2.375,9.025)$, $(4.875,38.025)$, $(9.875,156.025)$ from left to right for the coupling componenent.}
  \label{fig:exflux_cpl_stab}
\end{figure}

Figure \ref{fig:cpl2_stab} shows the stability regions defined by $\beta_+$ and
$d_+$ with fixed $\beta_-$ and $d_-$. The stability region for the explicit flux
coupling is larger than that obtained from the negative domain in Figure
\ref{fig:cpl_stab}. For some other values of $\beta_-$ and $d_-$, as shown in
Figure \ref{fig:exflux_cpl2_stab}, the stability region appears almost
unchanged. The reason is  that the negative domain is in a stable regime
$\beta_- \ll d_-$ and the effect of the upper right block of $\mathbf{B}$ to the
eigenvalues of $\mathbf{M}$ is insignificant.

All the experiements for explicit flux coupling show that the classical absolte
stability (independent of $d_\pm$) can be recovered in the limiting case
$\beta_-$ or $\beta_+$ approaches zero, as predicted by the analysis in Section
\ref{sec:exfluximts}.
\begin{figure}
  \centering
  \includegraphics[width=0.3\linewidth]{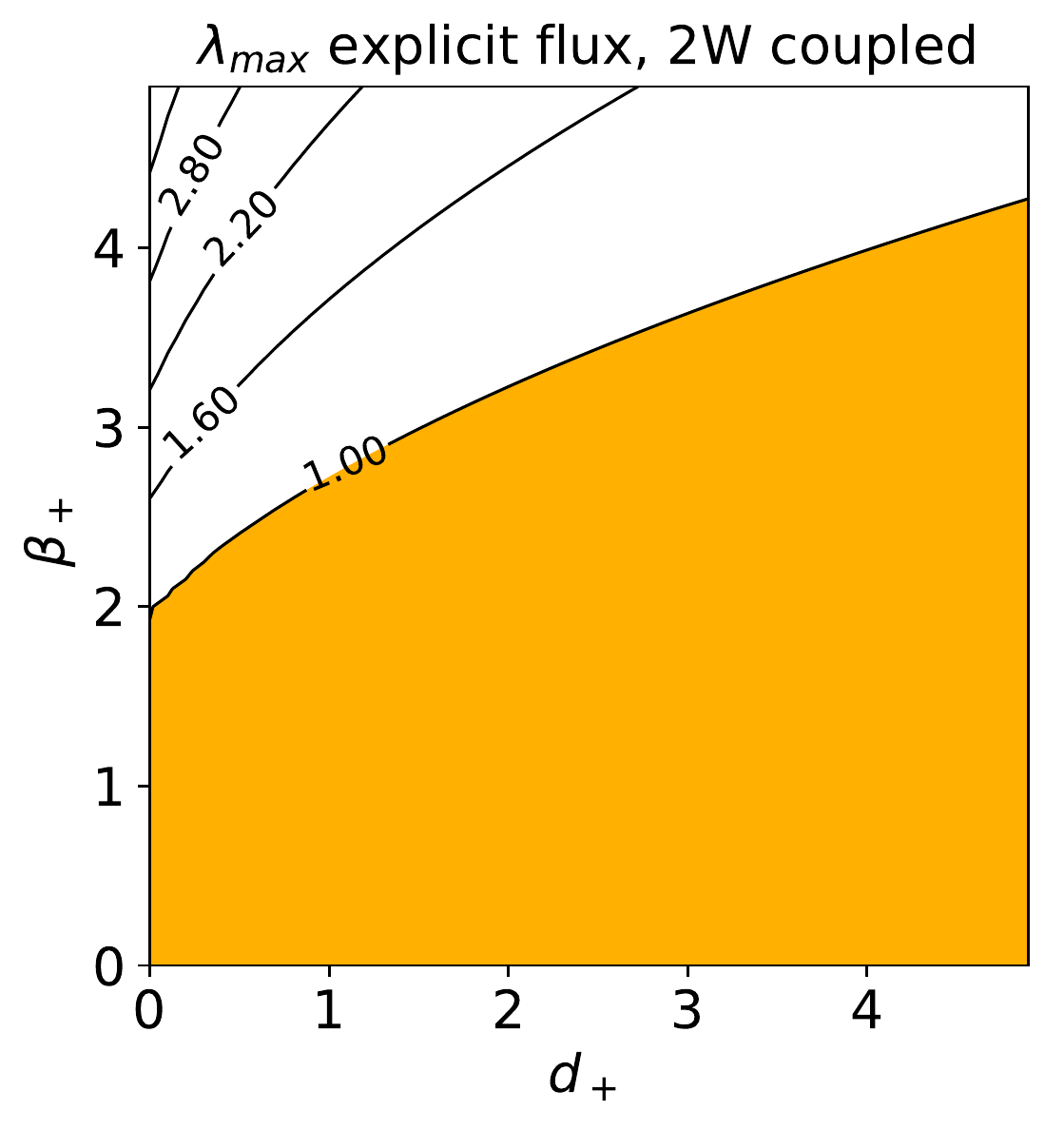}
  \includegraphics[width=0.33\linewidth]{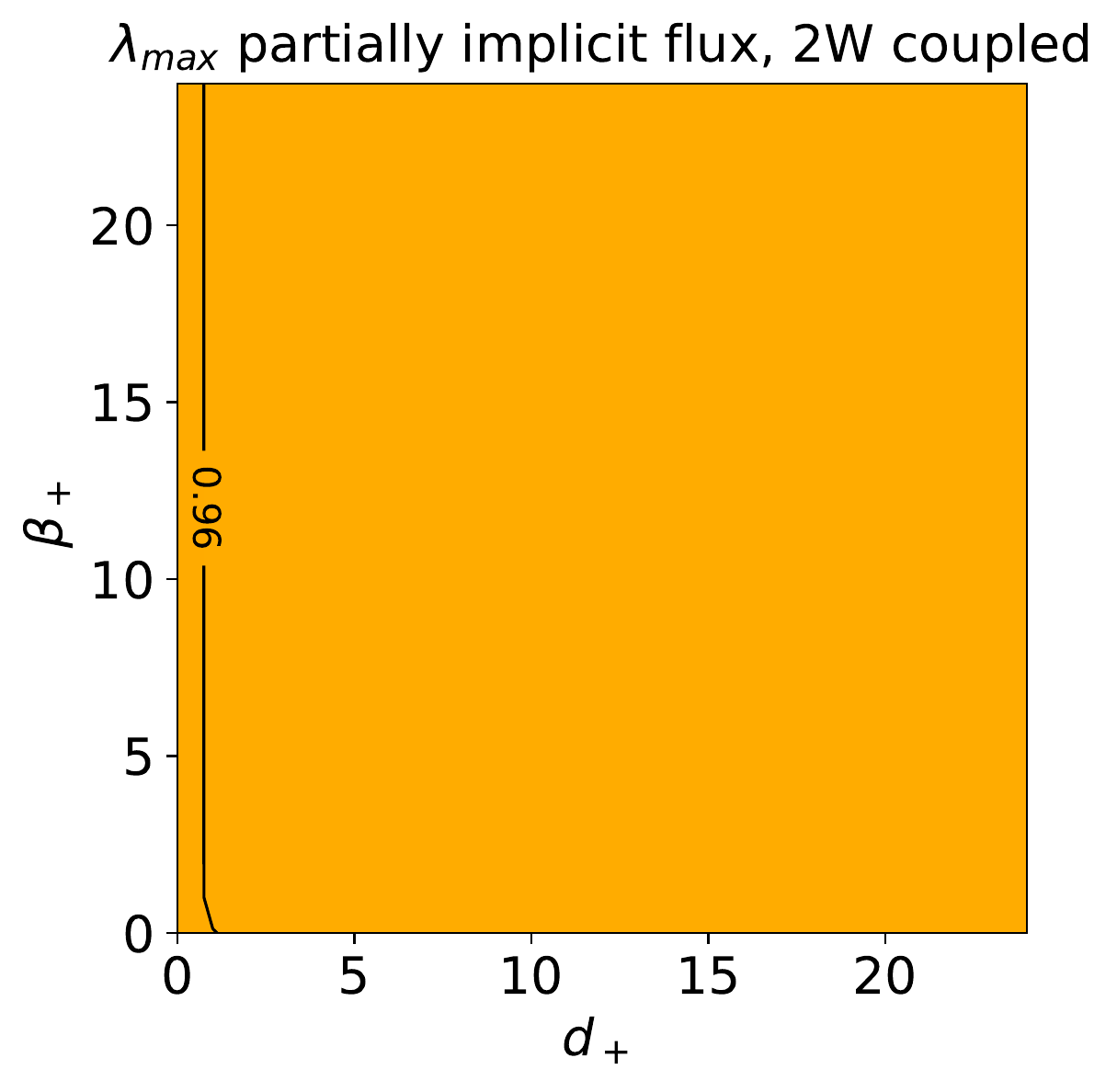}
  \includegraphics[width=0.308\linewidth]{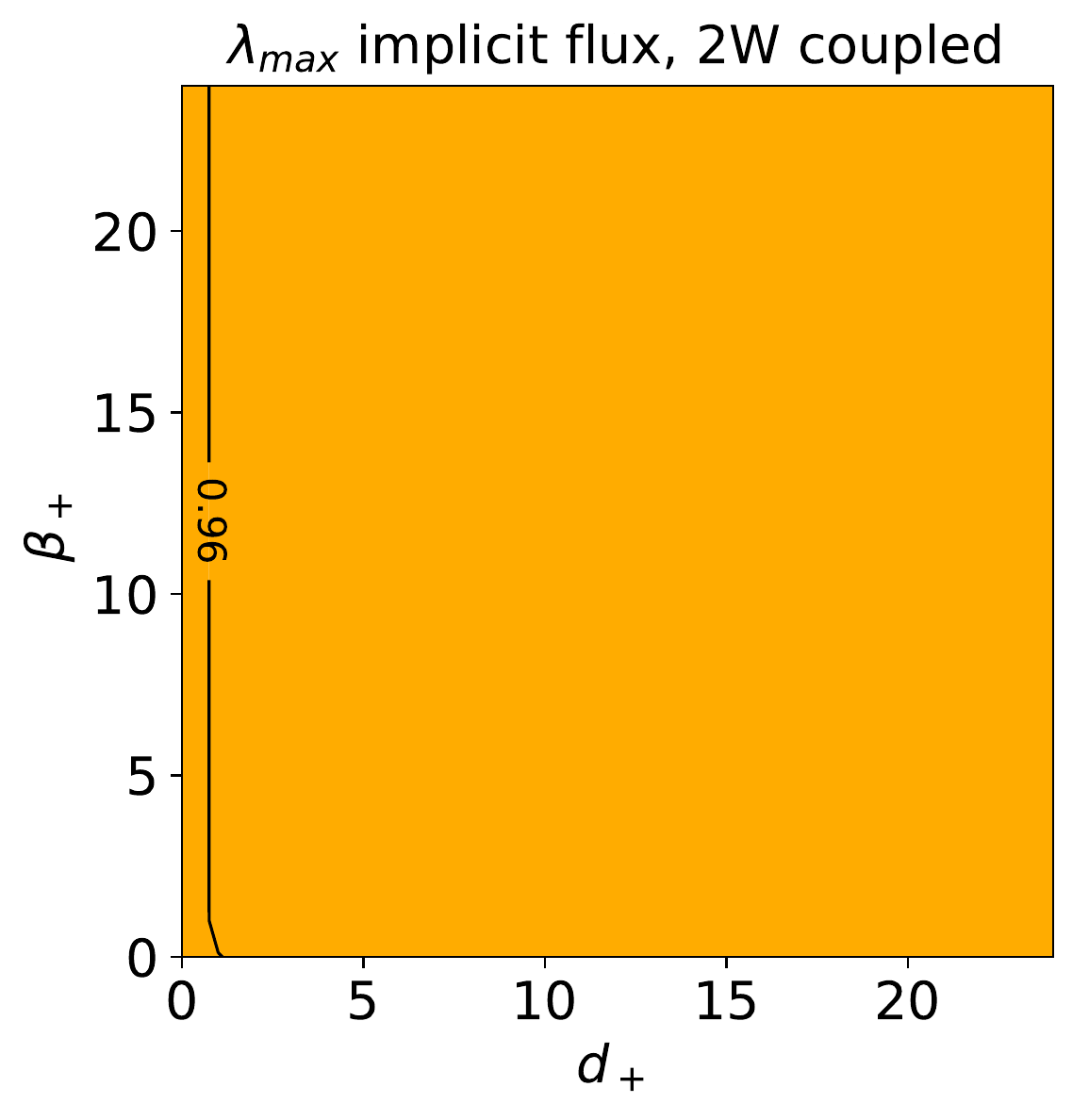}
  \caption{Stability regions for various flux coupling strategies with fixed parameters $\beta_-=0.005938$ and $d_-=9.025$ for the coupling component.}
  \label{fig:cpl2_stab}
\end{figure}
\begin{figure}
  \centering
  \includegraphics[width=0.32\linewidth]{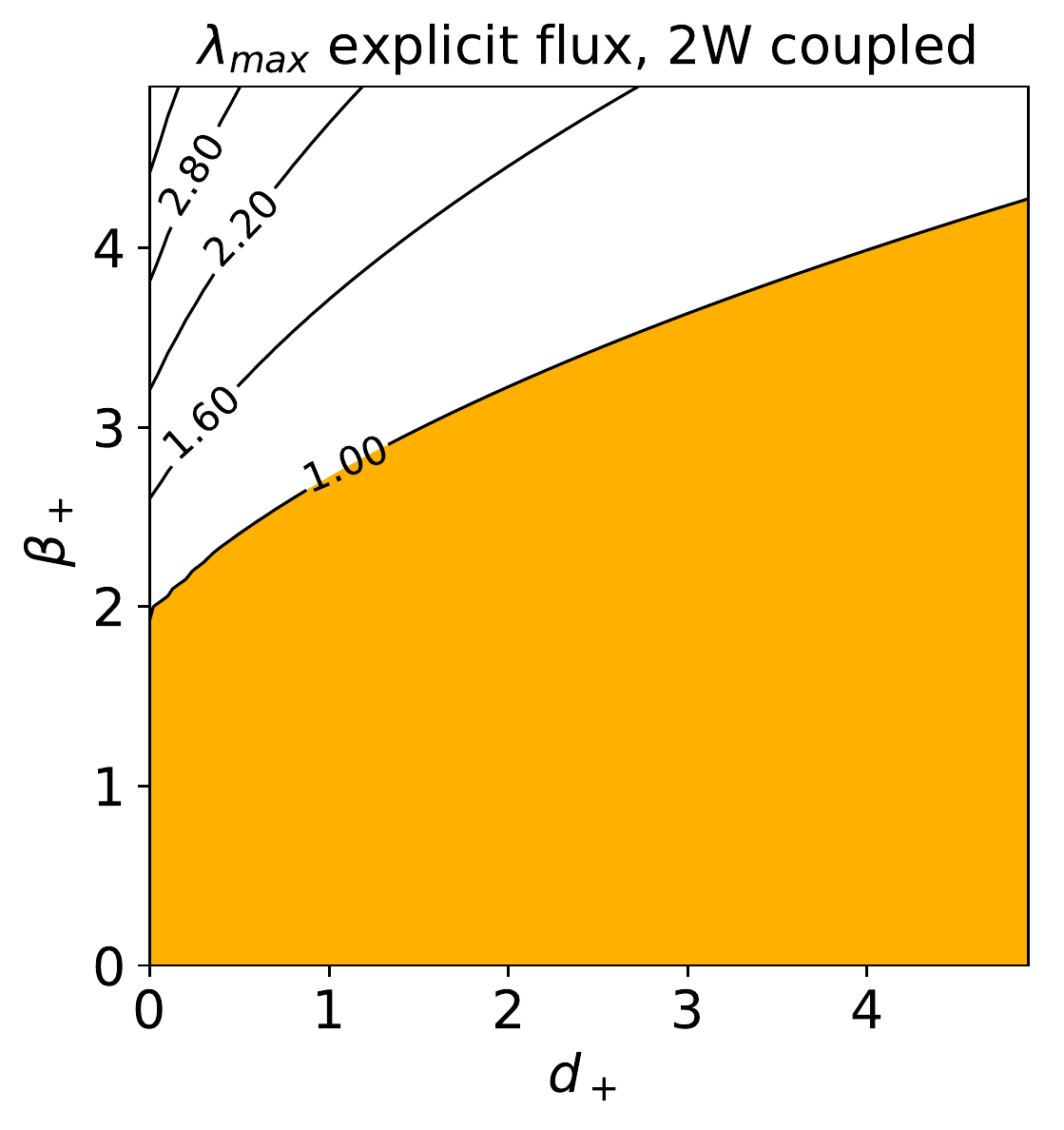}
  \includegraphics[width=0.32\linewidth]{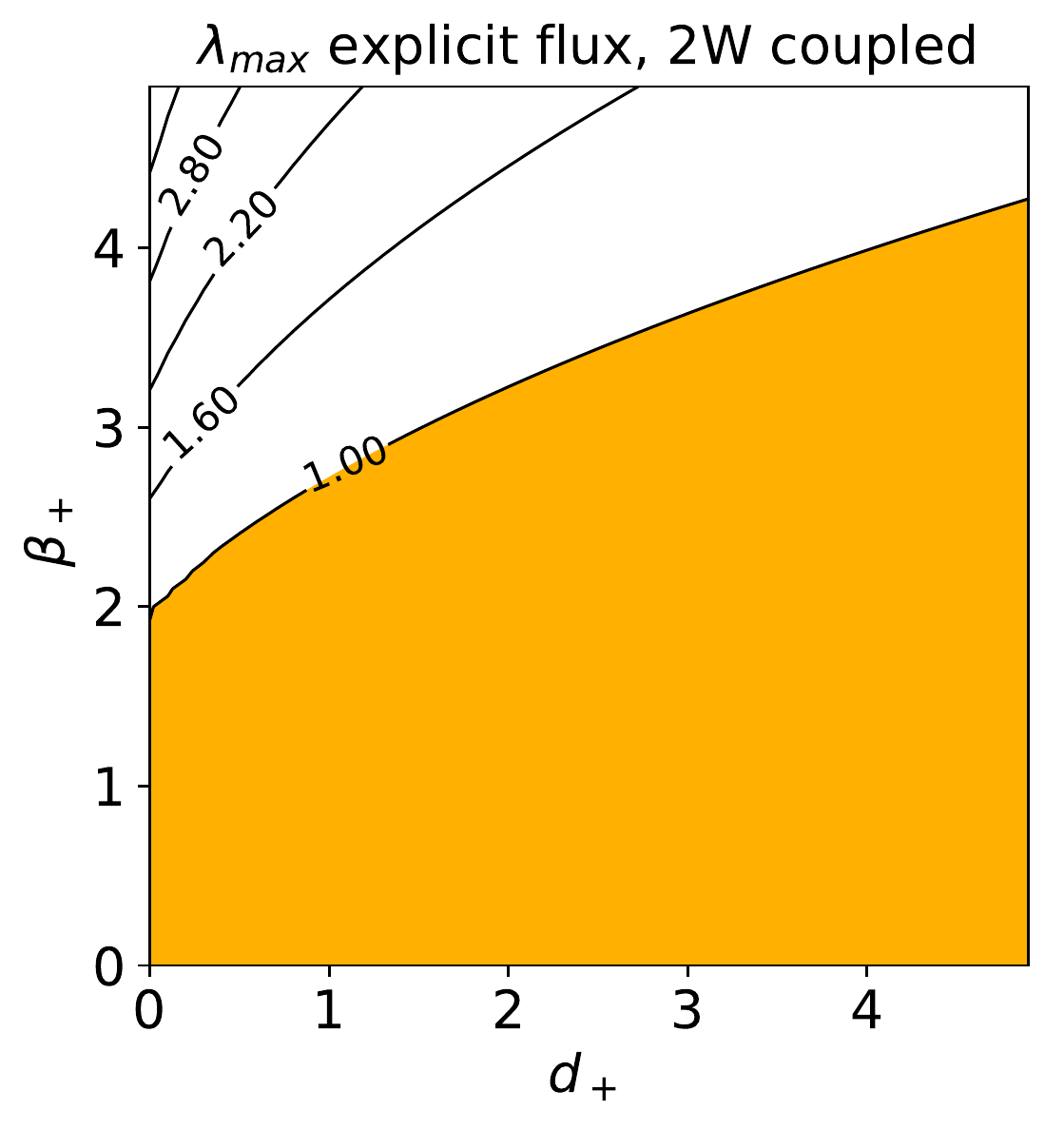}
  \includegraphics[width=0.32\linewidth]{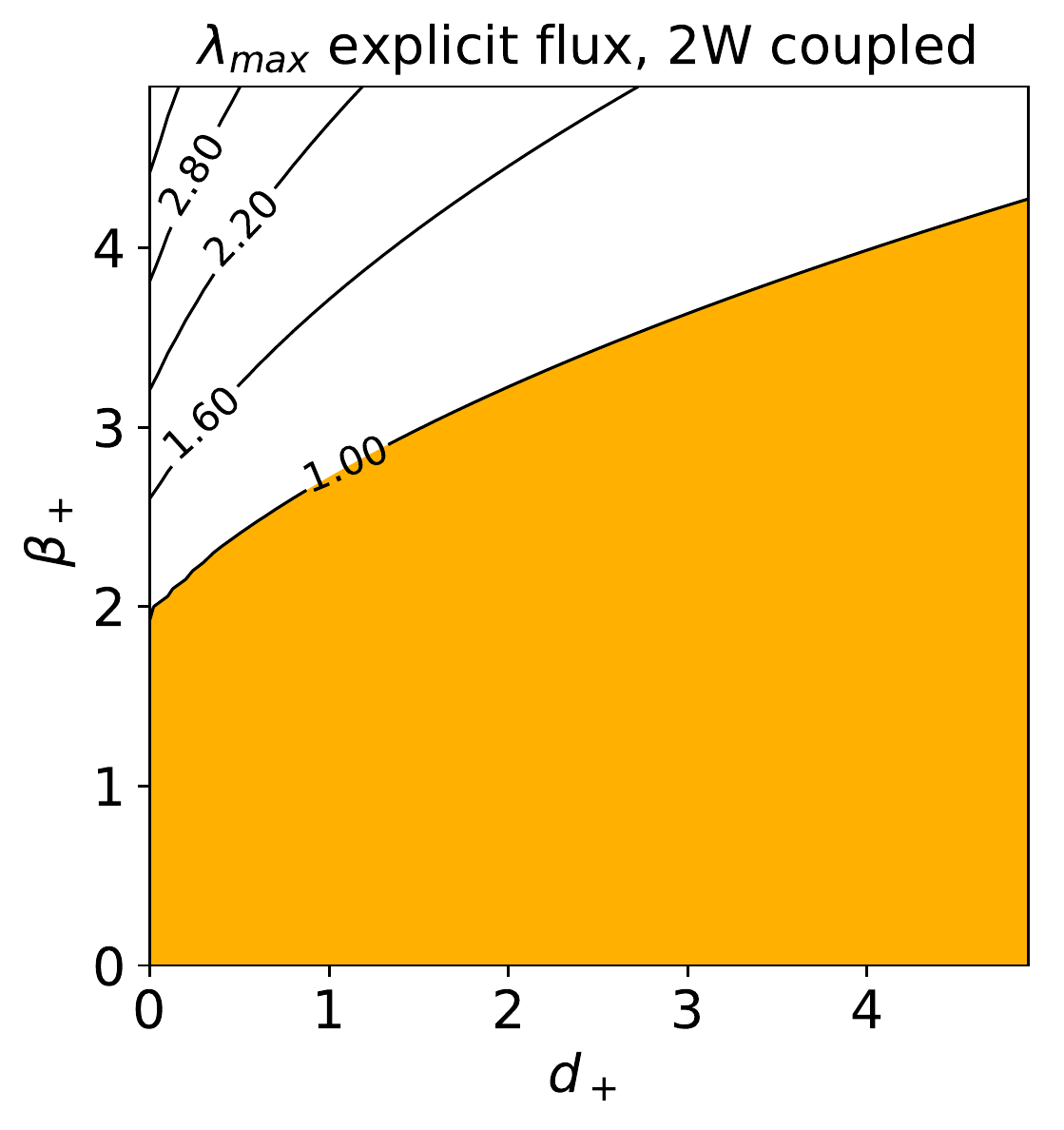}
  \caption{Stability regions for explicit flux coupling with fixed parameters $(\beta_-, d_-)$ set to  $(0.012188,38.025)$, $(0.024688,156.025)$, $(0.049688,632.025)$ from left to right for the coupling component.}
  \label{fig:exflux_cpl2_stab}
\end{figure}

\subsubsection{One-way coupled model}

Figure \ref{fig:exflux_stab_eigen} plots the stability regions for the one-way
coupled model according to the matrix stability analysis. It agrees with the
theoretical prediction depicted in Figure \ref{fig:exflux_stab}. Interestingly,
it also matches one of the two scenarios shown for the two-way coupled system
(see Figure \ref{fig:exflux_cpl2_stab}) but differs significantly from the other
scenario (see Figure \ref{fig:exflux_cpl_stab}). These results indicate that the
analysis for the one-way coupled model with an assumption of fixed boundary
$T_{\frac12}^n=0$ gives a good approximation to the coupled scenario when the
coupling component is in a stable regime where the bulk Courant number $\beta$
is much less than the diffusion Courant number $d$.
\begin{figure}
  \centering
  \includegraphics[width=0.4\linewidth]{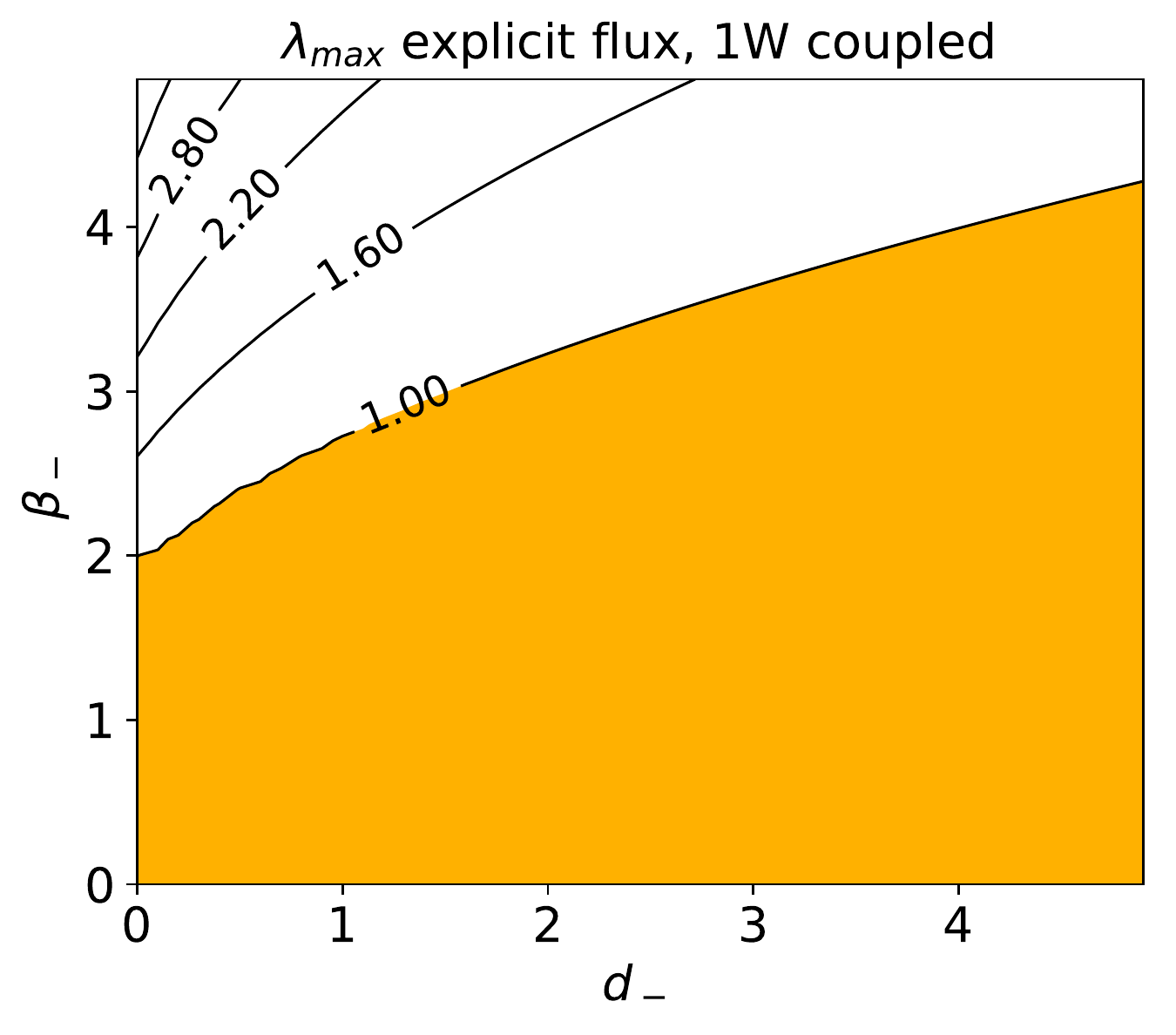}
  \includegraphics[width=0.4\linewidth]{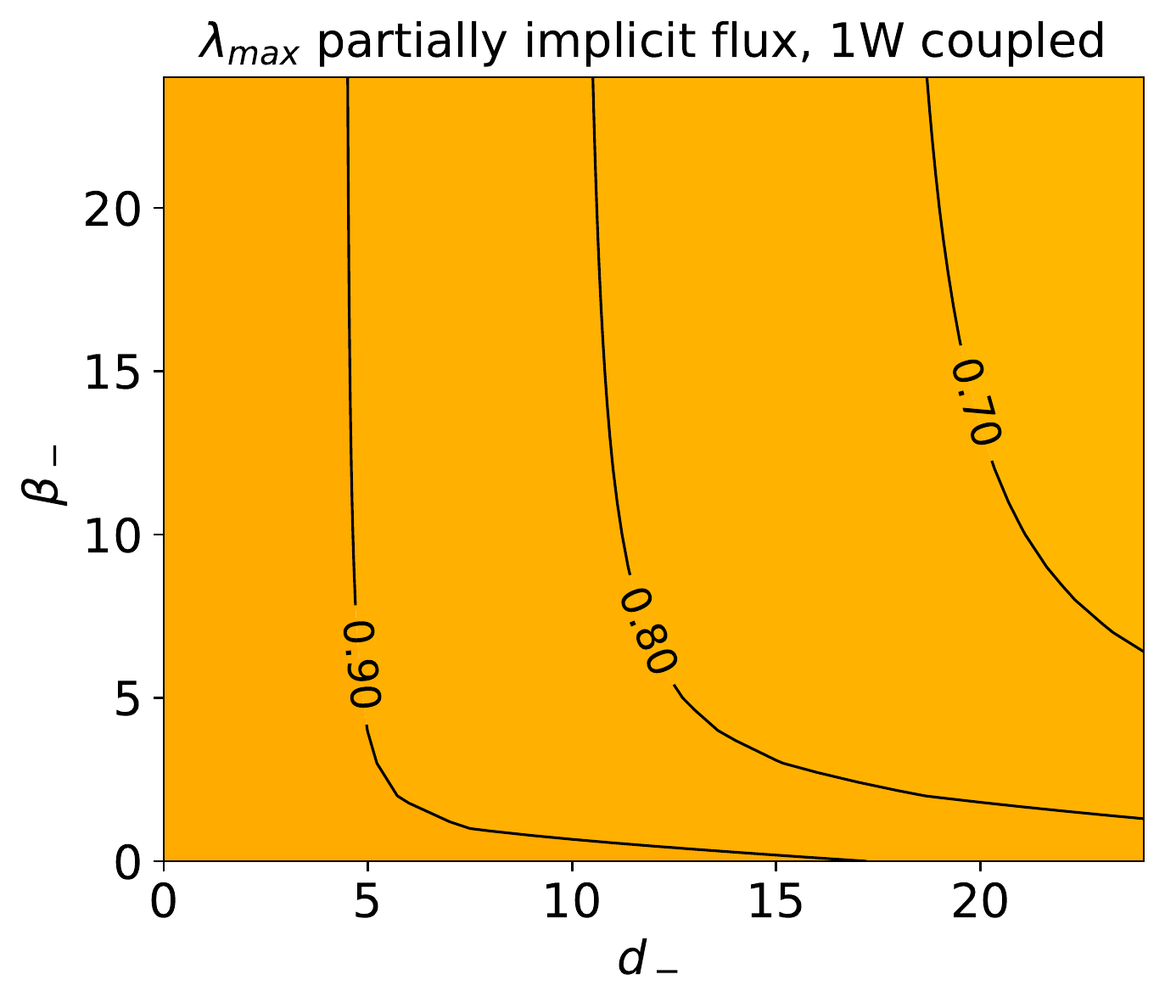}
  \caption{Stability regions for explicit flux coupling (left) and partially implicit flux coupling (right).}
  \label{fig:exflux_stab_eigen}
\end{figure}

Therefore, the analysis based on the one-way coupled model in Section
\ref{sec:forced_model} makes sense for atmosphere models forced by ocean models
but may not be ideal for the ocean model because the heat capacity of atmosphere
is much smaller and atmospheric dynamics evolves more rapidly, leading to a
small bulk Courant number $\beta$.

This observation can also be explained from the matrix analysis perspective. In
the case $(\theta,\gamma) = (0,0)$ and $\beta_-$ negligible, the stability
matrix can be denoted compactly as
\begin{equation}
\mathbf{M} =  \left[ \begin{array}{c c} A_{11} & 0 \\ 0 & A_{22}\end{array} \right]^{-1}
\left[ \begin{array}{c c} B_{11} & 0 \\ B_{21} & B_{22}\end{array} \right]
= \left[ \begin{array}{c c} A_{11}^{-1} B_{11} & 0 \\ A_{22}^{-1} B_{21} & A_{22}^{-1} B_{22}\end{array} \right] ,
\end{equation}
where $\mathbf{M} \in \R^{n\times n}, A_{11},B_{11} \in \R^{k\times k},
A_{22},B_{22} \in \R^{(n-k)\times (n-k)}, B_{21} \in \R^{(n-k) \times n}$. Then
the eigenvalues $\lambda$ of $\mathbf{M}$ must be roots of
$\textrm{det}(\mathbf{M}-\lambda \mathbf{I}_n) = 0$. We can see that
$\textrm{det}(\mathbf{M}-\lambda \mathbf{I}_n) = \textrm{det} (A_{11}^{-1}
B_{11} -\lambda \mathbf{I}_k) \; \textrm{det}(A_{22}^{-1}B_{22} -\lambda
\mathbf{I}_{n-k})$. Therefore, the spectrum of $\mathbf{M}$ is the union of the
spectrum of $A_{11}^{-1} B_{11}$ and that of $A_{22}^{-1} B_{22}$. If the
coupling domain (corresponding to $A_{22}^{-1} B_{22}$) is in the stable regime,
then the boundary of the stability region of the coupled system is determined by
the spectrum of $A_{11}^{-1} B_{11}$.

For general cases where $\beta_-$ is not negligible in the coupled system, one
can use the Schur complement to obtain $\textrm{det}(\mathbf{M}-\lambda
\mathbf{I}_n) = \textrm{det} (A_{11}^{-1} B_{11} -\lambda \mathbf{I}_k) \;
\textrm{det}(A_{22}^{-1}B_{22} -\lambda \mathbf{I}_{n-k} -
A_{22}^{-1}B_{21}(A_{11}^{-1} B_{11} -\lambda \mathbf{I}_k)^{-1} A_{11}^{-1}
B_{12})$ , if $B_{11}$ is invertible. Clearly, the spectral radius of
$\mathbf{M}$ depends on the spectral radiuses of $A_{11}^{-1} B_{11}$ and its
Schur complement matrix in $\mathbf{M}$. Further, given that $A_{11}^{-1}
B_{11}$ and $A_{22}^{-1}B_{22}$ are stable, $\mathbf{M}$ is not necessarily
stable. This result implies that the stability region for the one-way coupled
model is the best limiting case of stability regions for the two-way coupled
model.

\subsection{Dirichlet-Neumann condition}
The matrix form of the CSS algorithm (\ref{eq:css}) with explicit coupling flux
and the Dirichlet-Neumann condition exhibits a similar structure. For the purely
explicit method, it can be written as
\begin{subequations}
  \footnotesize
  \begin{align}
    \mathbf{A} = \nonumber\left[
    \resizebox{0.26\textwidth}{!}{$\displaystyle
    \begin{array}{c c c c c : c c c c c}
    1       &         &         &         &         &         &         &         &          &         \\
            & 1       &         &         &         &         &         &         &         &         \\
            &         & \ddots  &         &         &         &         &         &         &         \\
            &         &         & 1       &         &         &         &         &         &         \\
            &         &         &         & (1 + r)/2 &  &         &         &         &       \\
             \hdashline
            &         &         &         &   & 1&    &         &         &         \\
            &         &         &         &         &         & 1       &       &         &         \\
            &         &         &         &         &         &         & \ddots    &         &         \\
            &         &         &         &         &         &         &       &         & 1
    \end{array}
    $}
    \right]
    \mathbf{B} =
        \left[
    \resizebox{0.52\textwidth}{!}{$\displaystyle
    \begin{array}{c c c c c : c c c c c}
    1 - 2 d_- &  d_- &        &   &                       &                       &   &        &   &     \\
    d_-  & 1 - 2 d_-  &  d_-      &   &                       &                       &   &        &   &     \\
      &\ddots & \ddots & \ddots  &                       &                       &   &        &   &     \\
      &   &   d_-  & 1 - 2 d_-   & d_-                       &                       &   &        &   &     \\
      &   &        & d_- & (1+r)/2 - d_- - d_+ r &  d_+ r &   &        &   &     \\
           \hdashline
      &   &        &   & d_+    & 1-2 d_+  & d_+  &        &   &     \\
      &   &        &   &                       &               d_+      & 1-2 d_+  &  d_+      &   &     \\
      &   &        &   &                       &                       &   & \ddots &   &     \\
      &   &        &   &                       &                       &   &        &  d_+ & 1 - 2 d_+
    \end{array}
    $}
    \right],
  \end{align}
\end{subequations}
and the implicit algorithm \eqref{eq:implicit_css} corresponds to
\begin{subequations}
  \footnotesize
  \begin{align}
    &\mathbf{A} =
    &\nonumber\left[
    \resizebox{0.5\textwidth}{!}{$\displaystyle
    \begin{array}{c c c c c : c c c c c}
    2 d_-+1 & -d_-    &         &         &         &         &         &         &         &         \\
    -d_-    & 2 d_-+1 & -d_-    &         &         &         &         &         &         &         \\
            & \ddots  & \ddots  & \ddots  &         &         &         &         &         &         \\
            &         &    -d_- & 2 d_-+1 & -d_-    &         &         &         &         &         \\
            &         &         & -d_-    &  (1 + r)/2 + d_-  &         &         &         &         &       \\
             \hdashline
            &         &         &         &         &  d_+ + 1& -d_+    &         &         &         \\
            &         &         &         &         &  -d_+   & 2 d_++1 & -d_+    &         &         \\
            &         &         &         &         &         & \ddots  & \ddots  & \ddots  &         \\
            &         &         &         &         &         &         &    -d_+ & 2 d_++1 & -d_+    \\
            &         &         &         &         &         &         &         & -d_+    & 2 d_++1
    \end{array}
    $}
    \right]
    &\mathbf{B} =
        \left[
    \resizebox{0.3\textwidth}{!}{$\displaystyle
    \begin{array}{c c c c c : c c c c c}
    1 &   &        &   &                       &                       &   &        &   &     \\
      & 1 &        &   &                       &                       &   &        &   &     \\
      &   & \ddots &   &                       &                       &   &        &   &     \\
      &   &        & 1 &                       &                       &   &        &   &     \\
      &   &        &   & (1 + r)/2 - d_+ r     &               d_+ r   &   &        &   &     \\
           \hdashline
      &   &        &   &              d_+      &  1 - d_+ &   &        &   &     \\
      &   &        &   &                       &                       & 1 &        &   &     \\
      &   &        &   &                       &                       &   & \ddots &   &     \\
      &   &        &   &                       &                       &   &        & 1 &     \\
      &   &        &   &                       &                       &   &        &   & 1
    \end{array}
    $}
    \right].
  \end{align}
\end{subequations}
\iflongversion
Futhermore, these also fit into the general formulation
\eqref{eq:matrix_form_bulk} for bulk interface condition if we set
$(\theta,\gamma)=(0,0)$ and set the parameter $b=\nu_+/\Delta z_+$ so that
$\beta_- = d_+ r$.
\fi

We plot the stability regions for different values of $r$ in Figure
\ref{fig:exflux_cpl_stab_dn_euler} and Figure \ref{fig:exflux_cpl_stab_dn} and
verify our stability predictions in Section \ref{sec:dn}. We can see from Figure
\ref{fig:exflux_cpl_stab_dn_euler} that the stability region for the purely
explicit method is determined by the CFL condition and does not change as $r$
varies. Nonetheless, Figure \ref{fig:exflux_cpl_stab_dn} shows that the
stability region for the implicit method shrinks as $r$ increases and expands as
$r$ approaches to zero. Moreover, the stability region is infinite in the $d_+$
direction, and the boundary moves along the $d_-$ direction as $r$ varies,
indicating that instability can arise from the negative domain, at the interface
of which the Neumann condition is imposed, but not from the positive domain. In
practical ocean-atmosphere systems, $r$ is typically at the order of $10^{-4}$
or $10^{-5}$; thus it is reasonable to use the Neumann condition for ocean and
the Dirichlet condition for atmosphere theoretically.
\begin{figure}
  \centering
  \includegraphics[width=0.32\linewidth]{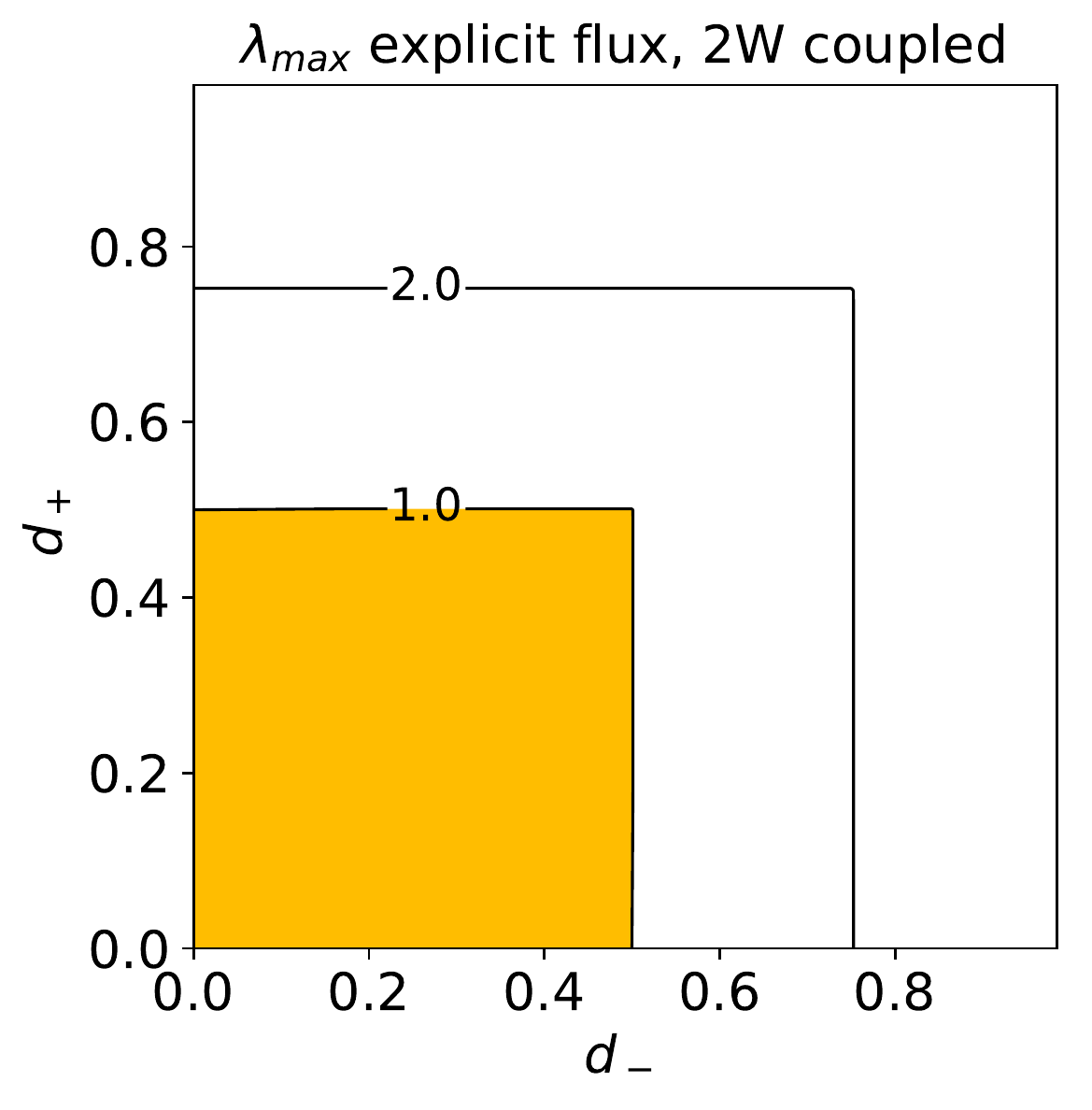}
  \includegraphics[width=0.32\linewidth]{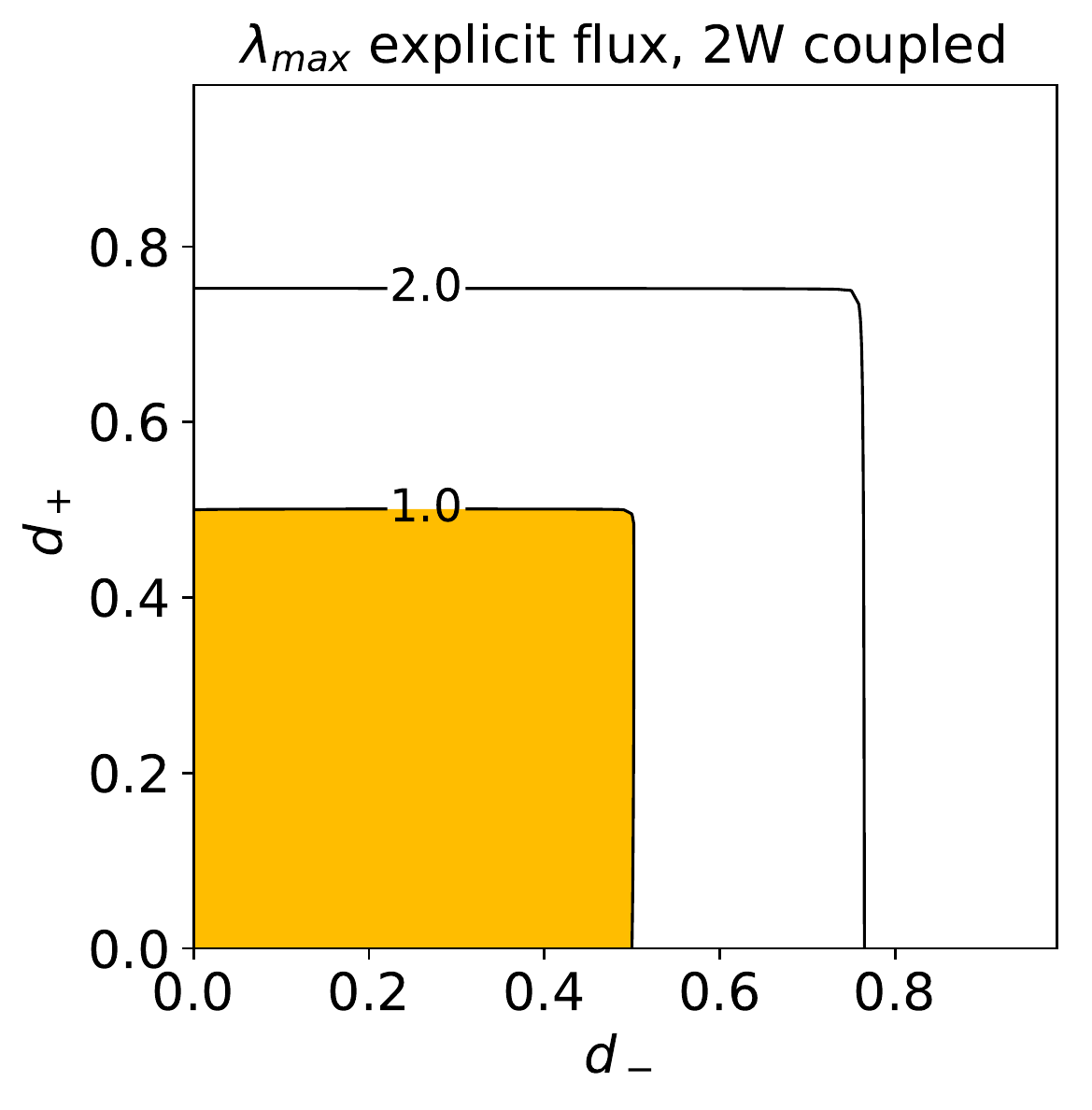}
  \includegraphics[width=0.32\linewidth]{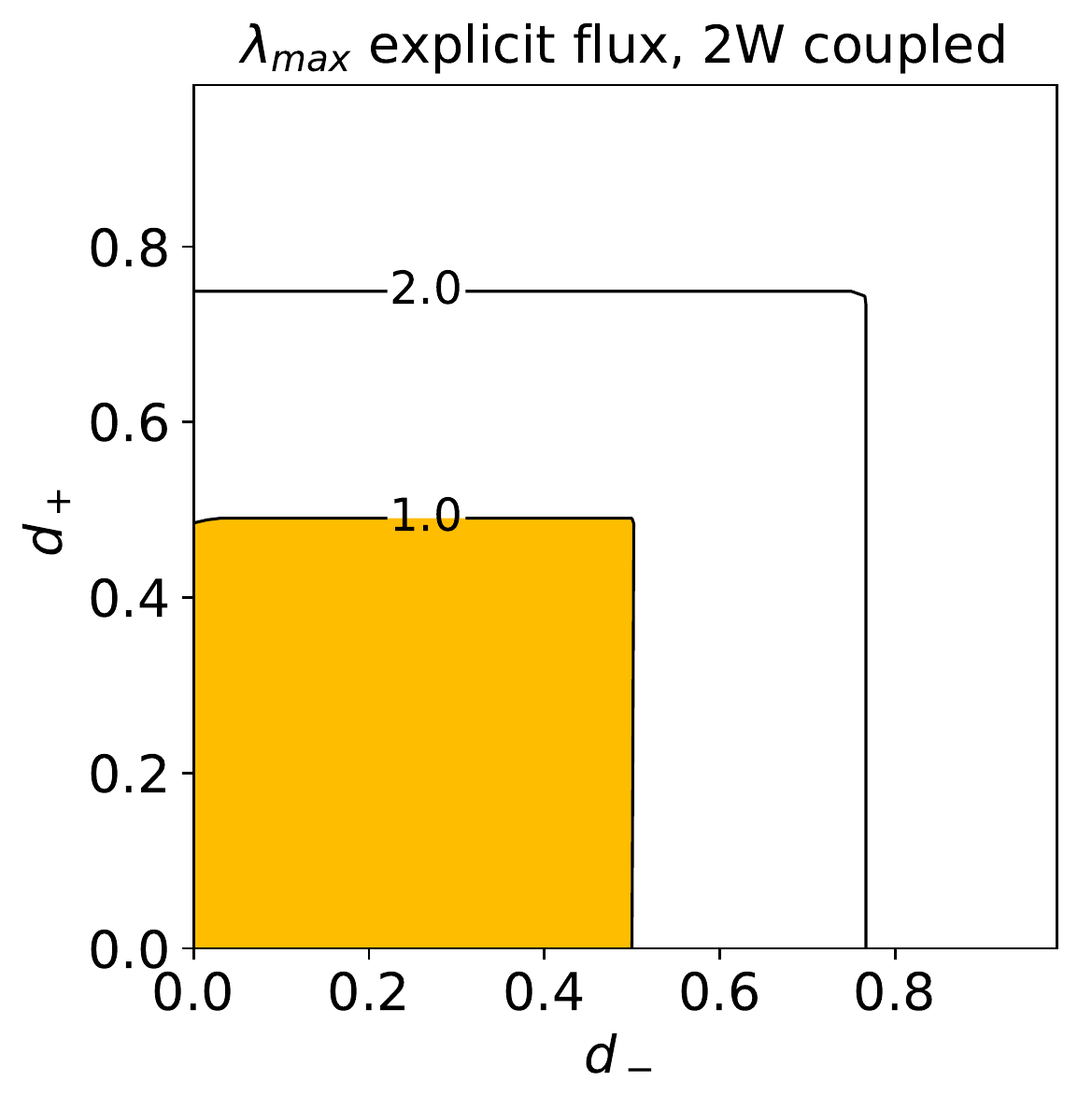}
  \caption{Stability regions for explicit integration methods with explicit flux coupling and Dirchlet-Neumann condition. From left to right, $r=2000,1,5e-4$ respectively.}
  \label{fig:exflux_cpl_stab_dn_euler}
\end{figure}
\begin{figure}
  \centering
  \includegraphics[width=0.32\linewidth]{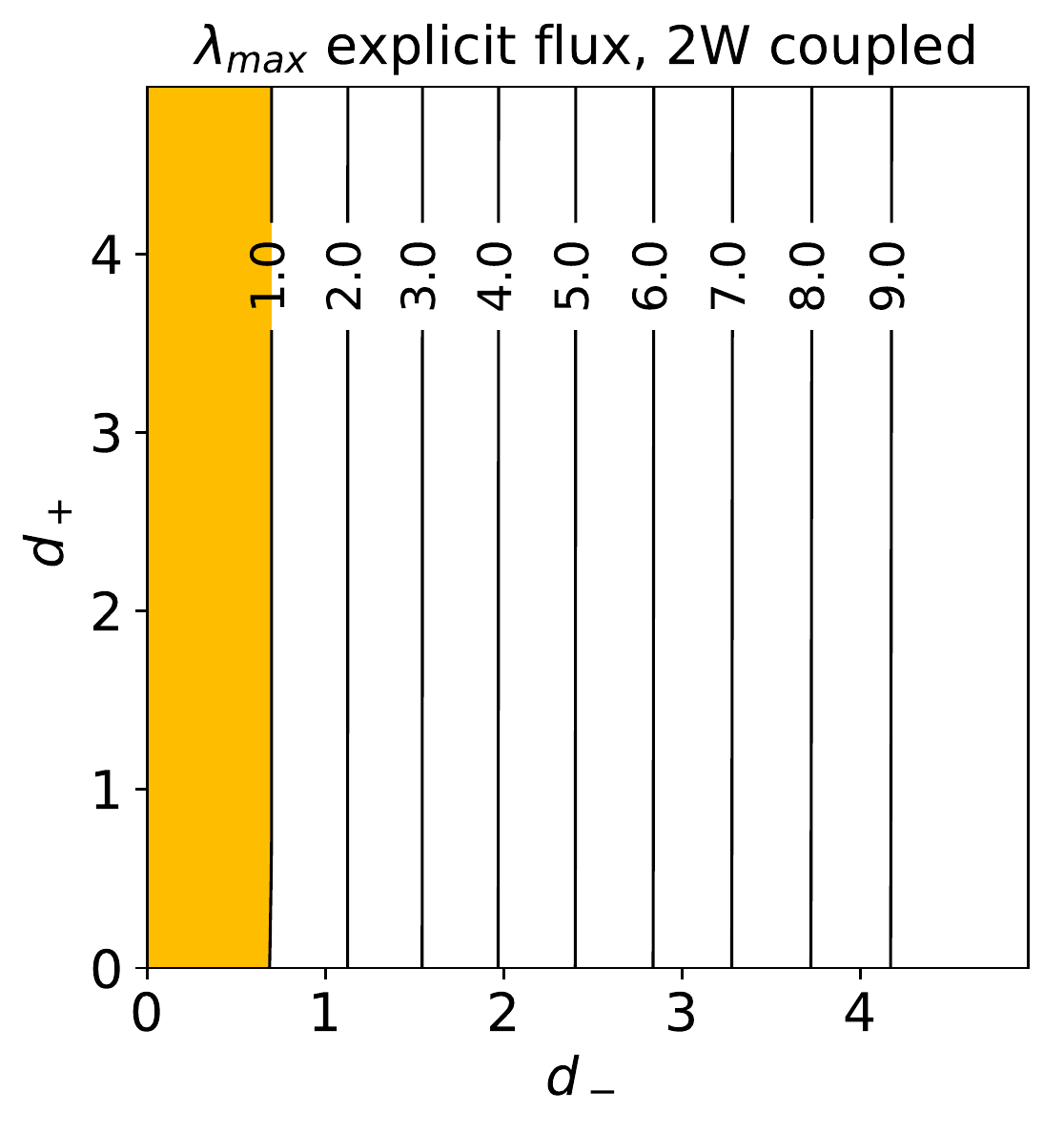}
  \includegraphics[width=0.32\linewidth]{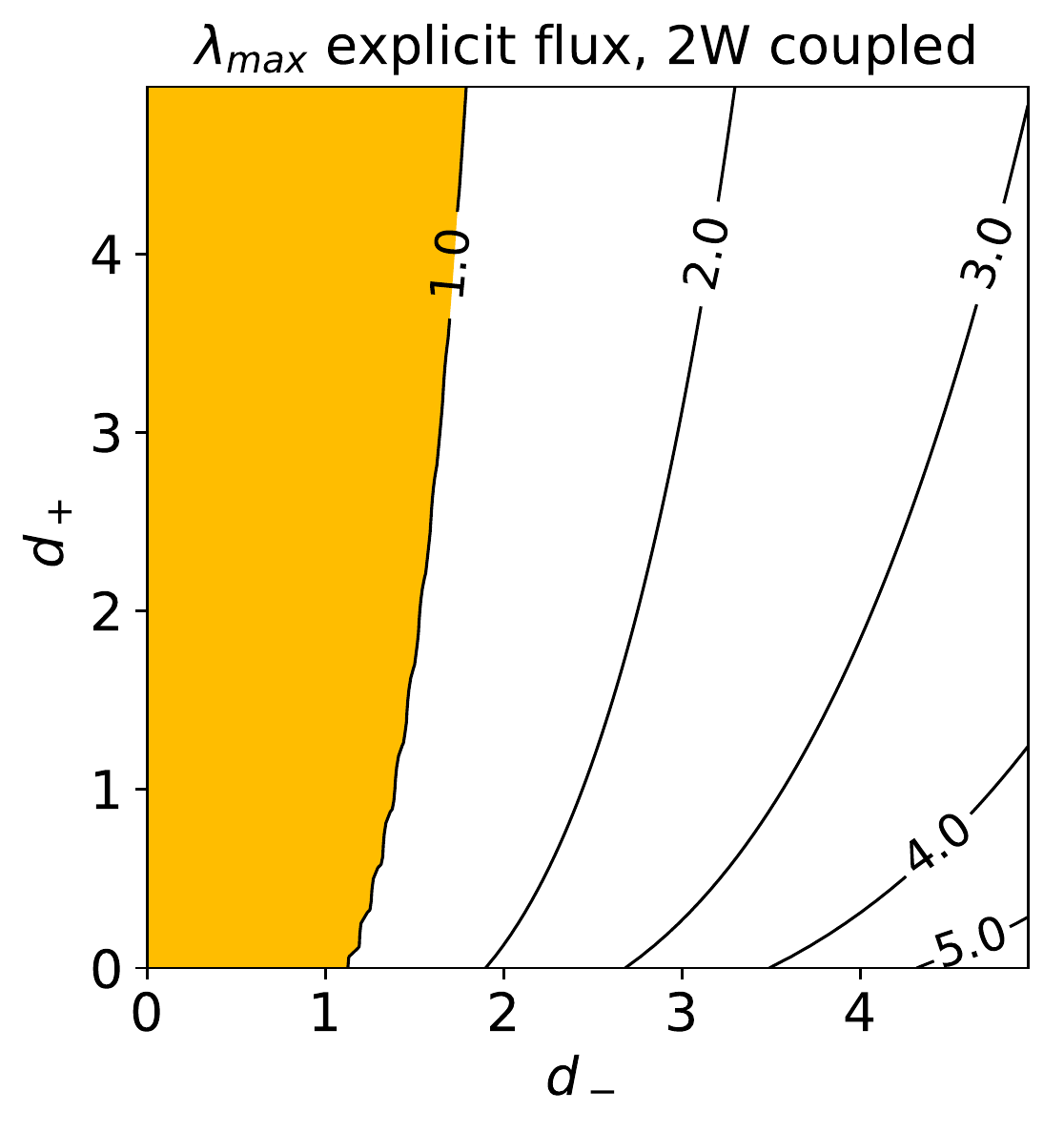}
  \includegraphics[width=0.32\linewidth]{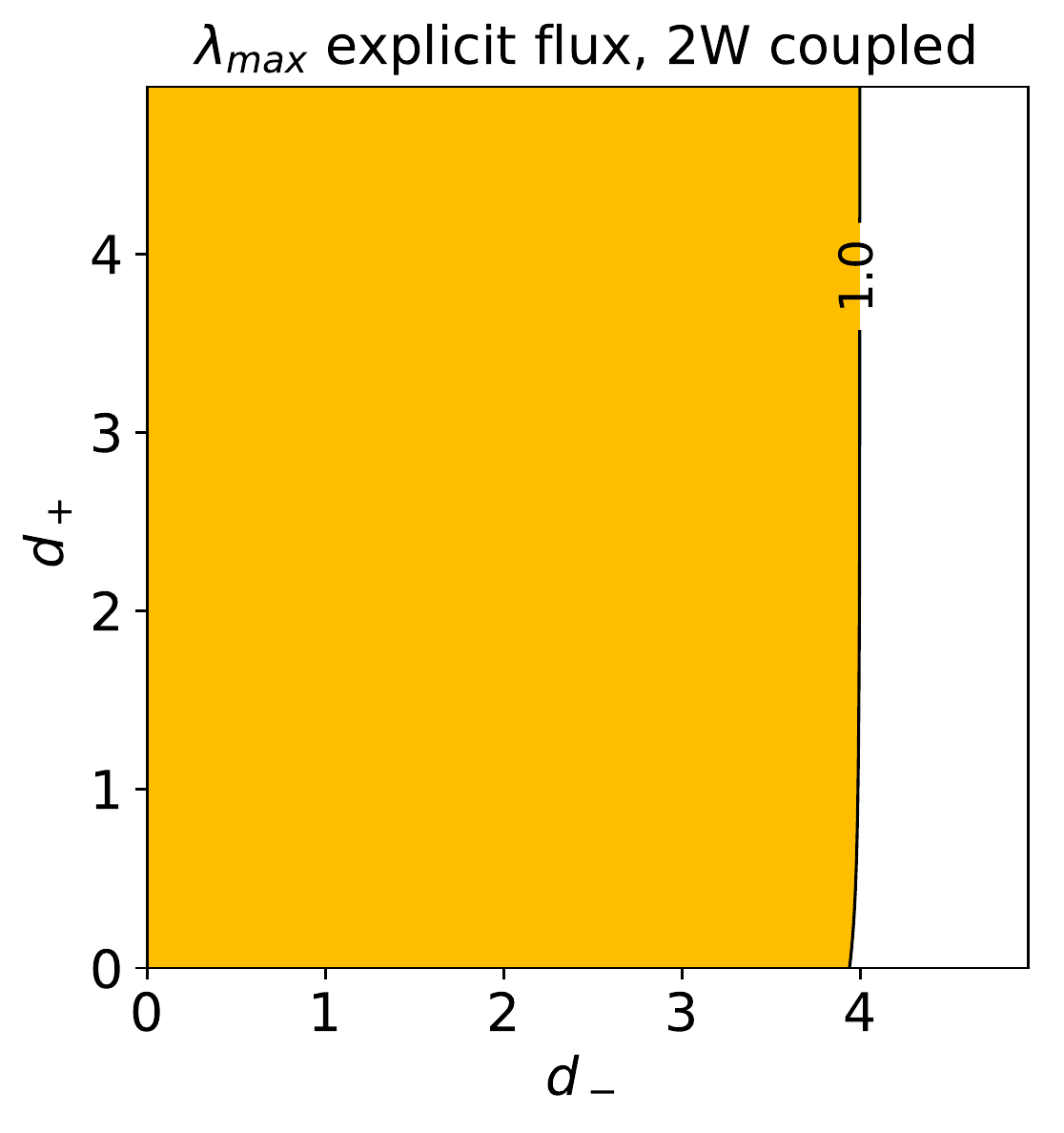}
  \caption{Stability regions for implicit integration methods with explicit flux coupling and Dirchlet-Neumann condition. From left to right, $r=2000,1,5e-4$ respectively.}
  \label{fig:exflux_cpl_stab_dn}
\end{figure}

\section{Conclusion}
The stability characteristics of common partitioned coupling algorithms for
ocean-atmosphere interactions have been studied with normal mode analysis and
matrix eigenvalue analysis. Three different flux coupling schemes with both
explicit and implicit time-stepping methods have been analyzed. Because of the
special modeling strategy of the interfacial physics in climate models, the bulk
interface condition is the focus of this work, although the classic
Dirichlet-Neumann condition that is widely used for FSI problems is also
included for completeness.

We show that the Dirichlet-Neumann condition imposed on the two coupled
components does not affect stability for a purely explicit scheme but becomes a
critical factor for an implicit method with explicit updating of the interfacial
flux. The influence of the bulk interface condition is characterized with a
variable that is formally similar to the Courant number. We show that the
coupled system is unconditionally stable when partially implicit time-stepping
methods are used for individual components and implicit flux coupling is used
for interface nodes. In addition, we derive CFL-like stability conditions for
the one-way coupled system and discuss the links between one-way and two-way
coupled systems. The theory is supported by numerical experiments based on
matrix eigenvalue analysis. Our results suggest that stability analysis of the
one-way coupled model should be used with caution: it is most effective when
used to reflect the stability behavior of the two-way coupled model when the
bulk Courant number of the coupling component is small (e.g., atmospheric models
forced by ocean models).

The results of the analysis performed for the 1D diffusion model is also
applicable to real 3D models, because the heat and turbulence in the
ocean-atmosphere circulation transfer mainly in the vertical direction. A
general circulation model can be considered to be a collection of many
single-column models, each of which may be associated with different bulk
Courant numbers due to variability in surface dynamics. Therefore, numerical
instability issues may arise in a subset of the models. Our analysis potentially
provides a framework to identify where instability may occur, and the results
can be used as guidelines for choosing proper flux coupling strategies and
developing new time-stepping methods with better stability properties.

\bibliography{hong}
\bibliographystyle{spmpsci}

\vfill
\begin{flushright}
\framebox{\parbox{0.58\textwidth}{The submitted manuscript has been created by
UChicago Argonne, LLC, Operator of Argonne National Laboratory (``Argonne'').
Argonne, a U.S. Department of Energy Office of Science laboratory, is operated
under Contract No. DE-AC02-06CH11357. The U.S. Government retains for itself,
and others acting on its behalf, a paid-up nonexclusive, irrevocable worldwide
license in said article to reproduce, prepare derivative works, distribute
copies to the public, and perform publicly and display publicly, by or on behalf
of the Government. The Department of Energy will provide public access to these
results of federally sponsored research in accordance with the DOE Public Access
Plan. http://energy.gov/downloads/doe-public-access-plan. }}
\normalsize
\end{flushright}

\end{document}